\documentclass[12pt]{article}
\usepackage{amssymb}
\usepackage{mathrsfs}
\usepackage{bm}
\usepackage{amsfonts}
\usepackage{graphicx}
\usepackage{setspace}
\usepackage{subfig}
\usepackage{indentfirst}
\usepackage{colortbl,dcolumn}
\usepackage{color}
\usepackage{comment}
\usepackage{xcolor}
\usepackage{amsmath}
\usepackage{bbm} 
\usepackage{dsfont} 
\usepackage{psfrag}
\usepackage{booktabs}
\usepackage{array}
\usepackage{cite}
\usepackage{amsthm}
\numberwithin{equation}{section}
\usepackage[top=1in, bottom=1in, left=0.8in, right=0.8in]{geometry}
\newcommand{\R}{\mathbb{R}}
\newcommand{\N}{\mathbb{N}}
\newcommand{\E}{\mathbb{E}}
\renewcommand{\P}{\mathbb{P}}

\newtheorem{thm}{Theorem}[section]

\newtheorem{lem}[thm]{Lemma}
\newtheorem{prop}[thm]{Proposition}
\newtheorem{cor}[thm]{Corollary}
\newtheorem{rem}[thm]{Remark}
\newtheorem{example}[thm]{Example}
\newtheorem{assumption}[thm]{Assumption}
\allowdisplaybreaks[4] 

\begin{document}
\title{
Uniform-in-time weak error estimates of explicit full-discretization schemes for SPDEs with non-globally Lipschitz coefficients
}
                
\author{Yingsong Jiang,  Xiaojie Wang 
\thanks{
This work was supported by Natural Science Foundation of China (12471394, 12071488, 12371417).
E-mail addresses:
	x.j.wang7@csu.edu.cn (Corresponding author), yingsong@csu.edu.cn.  
} 
\\
\footnotesize School of Mathematics and Statistics, Hunan Research Center of the Basic Discipline for Analytical  
\\
\footnotesize Mathematics, HNP-LAMA, Central South University, Changsha 410083, Hunan, China
}
\maketitle

\begin{abstract}

This article is devoted to long-time weak approximations of stochastic partial differential equations (SPDEs) evolving in a bounded domain $\mathcal{D} \subset \mathbb{R}^d$, $d \leq 3$, with non-globally Lipschitz and possibly non-contractive coefficients.
Both the space-time white noise in one dimension $d=1$ and the trace-class noise in multiple dimensions $d=2,3$ are examined for the considered SPDEs. 
Based on a spectral Galerkin spatial semi-discretization, we propose a class of novel full-discretization schemes, which are explicit, easily implementable and preserve the ergodicity of the original dissipative SPDEs with possibly non-contractive coefficients. The uniform-in-time weak approximation errors are carefully analyzed in a low regularity and possibly non-contractive setting, with uniform-in-time weak convergence rates obtained. A key ingredient is to establish the uniform-in-time moment bounds (in $L^{4q-2}$-norm, $q \geq 1$) for the proposed fully discrete schemes in a super-linear setting. This is highly non-trivial due to the explicit full-discretization, the super-linearly growing coefficients and the low regularity of SPDEs. New arguments are elaborated to overcome the challenging problem, by fully exploiting a new contractive property of the semi-group in $L^{4q-2}$, the dissipativity of the nonlinearity and a particular benefit of the taming strategy. Numerical experiments are finally reported to verify the theoretical findings.

	\par 
	{\bf AMS subject classification:} {\rm\small 60H35, 60H15, 65C30.}	\\
	
	{\bf Keywords:} 
SPDEs with non-globally Lipschitz coefficients,
 explicit full discretization schemes,
 uniform-in-time moment bounds,
 uniform-in-time weak convergence rates,
 approximations of invariant measure

\end{abstract}

\section{Introduction}
Stochastic partial differential equations (SPDEs) have emerged as a class of mathematical models in various scientific areas, ranging from phase field dynamics to fluid mechanics. In the last decades, numerous works have been dedicated to strong and weak approximations of SPDEs over finite-time horizons (see, e.g., \cite{antonopoulou2021numerical,Yan2024IMA_posteriori,kruse2023bdf2,wang2014DCDSweak,becker2023strong,anton2020fully,andersson2016weak,CUI2021weak,Chen_Dang_Hong2024adaptive_IMA,wang2020efficient,Cai2021weak4ACE,Arnulf_strong4non_global_lips,Brehier_Allen_C,Feng_AC,MajeeProhl_AC,QiWang2019optimal,brehier2024SiamTV,zhang2025weak,HuangC_ShenJ2023Mathcomp}, to just mention a few).
Nevertheless, long-time approximations of SPDEs are of significant interest in many scenarios such as sampling from the invariant measure, and in this situation, the long-time convergence turns out to be indispensable.
As opposed to the finite-time approximation,
just a few works (e.g., \cite{brehier2014PA,brehier2022ESAIM,brehier2024FCM,wangyibo2024ACE,CUI2021weak}) analyzed long-time approximations of SPDEs, which is far from being well understood, particularly for SPDEs with superlinearly growing and non-contractive coefficients.

In this paper, we delve into long-time explicit approximations of parabolic SPDEs in the Hilbert space $H := L^2(\mathcal{D}; \R)$ of the form:
\begin{equation}
\label{eq:SEE(in_Introduction)[weak4AC-25]}
\left\{\begin{array}{l}
\mathrm{d} X(t)  = - A X(t) \, \mathrm{d} t + F(X(t)) \, \mathrm{d} t + \, \mathrm{d} W(t), 
\quad t > 0, \\
X(0) = X_0.
\end{array}\right.
\end{equation}
Here, $\mathcal{D} \subset \R^d$, $d \leq 3$ is a bounded spatial domain with smooth boundary, $-A$ is the Laplacian operator with homogeneous Dirichlet boundary conditions, $F$ is a nonlinear Nemytskii operator associated with a real-valued function $f \colon \R \rightarrow \R$, i.e., 
$
    F(u)(x) := f( u(x) ), 
    x \in \mathcal{D},
$
and $\{W(t)\}_{t \geq 0}$ is an $H$-valued (possibly cylindrical) $Q$-Wiener  process (see Assumptions \ref{assump:A(linear_operator)[weak4AC-25]}-\ref{assump:F(Nonlinearity)[weak4AC-25]} below for details).
%

When the nonlinear mapping $F$ is globally Lipschitz, Br{\'e}hier \cite{brehier2014PA} used 
the classic linear implicit Euler scheme to approximate the invariant measure of \eqref{eq:SEE(in_Introduction)[weak4AC-25]} with space-time white noise. The uniform-in-time weak convergence rates were also revealed there.
Very recently, a modified regularity-preserving Euler scheme was proposed and analyzed in \cite{brehier2024FCM} for similar problems in a globally Lipschitz setting, with total variation error bounds obtained.  For Allen-Cahn type SPDEs with polynomially growing $F$, the authors of \cite{CUI2021weak} used the (nonlinearity implicit) backward Euler method for long-time approximations of \eqref{eq:SEE(in_Introduction)[weak4AC-25]} and derived uniform-in-time weak convergence rates of the fully discrete schemes. Later on, an explicit tamed exponential Euler scheme was employed in \cite{brehier2022ESAIM} for approximation of the invariant distribution of SPDEs \eqref{eq:SEE(in_Introduction)[weak4AC-25]}, where error bounds have a polynomial dependence on the time length $T$.
This happened because uniform-in-time moment bounds cannot be derived for the tamed scheme proposed in \cite{brehier2022ESAIM}. In a more recent preprint \cite{wangyibo2024ACE}, the authors did not discretize the stochastic convolution in the temporal direction and proposed a kind of tamed accelerated exponential Euler method for weak approximations of SPDEs \eqref{eq:SEE(in_Introduction)[weak4AC-25]} in the one-dimensional case $d=1$. Uniform-in-time moment bounds  were derived there for the scheme, and uniform-in-time weak convergence rates were revealed in a contractive setting.

In this work, we propose a class of novel, explicit full-discretization schemes of exponential type for SPDEs \eqref{eq:SEE(in_Introduction)[weak4AC-25]} in dimensions $d \leq 3$.
Based on a spectral Galerkin spatial semi-discretization, the proposed exponential integrators (see Section \ref{sec:scheme[weak4AC-25]} for details) read as
\begin{equation}\label{eq:intro_full-discre[weak4AC-25]}
    X^{N,\tau}_{t_{m+1}}
    = 
    E_N(\tau) X^{N,\tau}_{t_m} 
    + 
    \tau E_N(\tau)
    P_N F_{\tau, N}
    \big( X^{N,\tau}_{t_m} \big)
    +
    E_N(\tau) P_N \Delta W_m,
    \quad
    X^{N,\tau}_{0} = P_N X_0,
\end{equation}
where
\begin{equation}
    F_{\tau,N}(u)(x)
        : =
        f_{\tau,N}( u(x) ), 
        \quad
        x \in \mathcal{D},
\end{equation}
with
$f_{\tau,N}\colon \R \rightarrow \R$ being
a modification of $f \colon \R \rightarrow \R$, suggested as follows:
\begin{equation}
       f_{\tau,N}( v )
       :=
       \frac{ f( v ) }
       { \Big(  1 
            + 
            (\beta_1 \tau^{\theta} +  \beta_2 \lambda_N^{-\rho} )
          |v|^{ \frac{2q-2}{\alpha} } \Big)^\alpha
        },
        \quad
         \alpha \in (0, 1], 
         \
         \theta,
       \rho,
      \beta_1, \beta_2 >0.
    \end{equation}
%
%
Here, $\tau$ is the uniform time step-size, $\lambda_N , N \in \N$ is the $N$-th eigenvalue of the linear operator $A$, and one can think of $f$ as a polynomial of odd degree $2q-1$, $q \in \N$ with a negative leading coefficient. 
Clearly, the proposed schemes are explicit and easily implementable.
Different from the existing taming schemes \cite{wang2020efficient,brehier2022ESAIM,wang2024IMA,liu-shen_2025geometric,HuangC_ShenJ2023Mathcomp}, the newly proposed scheme incorporates a new term $\beta_2 \lambda_N^{-\rho}$ in the taming factor of a flexible degree $\alpha$ (instead of a fixed degree $\tfrac12$ or $1$ in the literature). In this setting, $f_{\tau,N}(u)$ has a linear growth \eqref{eq:assume_|F_tau(u)|_lambda[weak4AC-25]}, i.e., 
$
| f_{\tau,N}(u) |
\leq
    \widetilde{c}_3
    ( 1 + |u|
      +
      \lambda_N^{ \alpha \rho }|u| ).
$
Then one can choose small $\alpha$ such that
$\alpha \rho < 1- \tfrac{d}{4} $, to ensure that the linear growth constant of $f_{\tau,N}(u)$, i.e., $\lambda_N^{ \alpha \rho }$, can be properly controlled even in multiple dimensions $d =2, 3$ (see uniform bounds of $\mathcal{R}^{N,\tau}$ in the proof of Theorem \ref{thm:X^N,tau(bound)[weak4AC-25]}). 
%
%
%
%
To carry out the long-time weak error analysis, one needs to establish the uniform-in-time moment bounds (in $L^{4q-2}$-norm, $q \geq 1$) for the proposed fully discrete schemes in the super-linear setting. This turns out to be highly non-trivial and challenging due to the explicit full-discretization and the low regularity of SPDEs \eqref{eq:SEE(in_Introduction)[weak4AC-25]} in multiple dimensions $d\leq 3$. By fully exploiting a contractive property of the semi-group in $L^{4q-2}$ (see Proposition \ref{prop:E_N(t)_contractivity[weak4AC-25]}), the dissipativity of the nonlinearity, and the particular benefit of the new taming strategy (e.g., \eqref{eq:assume_|F_tau(u)|_lambda[weak4AC-25]}), as mentioned before, new arguments are elaborated to obtain the desired uniform-in-time moment bounds of the proposed fully discrete schemes (see Theorem \ref{thm:X^N,tau(bound)[weak4AC-25]}), without any constraint on the time-space stepsize ratio.

Equipped with the uniform moment bounds, we are able to do the long-time weak error analysis without relying on the Malliavin Calculus and achieve the desired uniform-in-time weak convergence rates in both the spatial and temporal directions (Theorem \ref{thm:uniform_weakerror[weak4AC-25]}), for
SPDEs with both contractive and non-contractive coefficients (Assumption \ref{assume:contractive_or_non-degeneracy[weak4AC-25]}).
Both the space-time white noise in one dimension $d=1$ and the trace-class noise in multiple dimensions $d\leq 3$ are examined.
%
%
%
Finally, we show the geometric ergodicity of the proposed full-discretization schemes (Proposition \ref{prop:invariant_measure_approximation[weak4AC-25]}).

%
It is also worthwhile to mention that, just when the present manuscript was almost finished, we were aware of a preprint \cite{liu-shen_2025geometric} that proposed ergodic tamed time-stepping schemes for SPDEs with multiplicative and trace-class noises, where only finite-time moment bounds and finite-time strong convergence rates are proved for the schemes.

The remainder of this article is organized as follows. In the next section we introduce the considered SPDEs. 
Section \ref{sec:scheme[weak4AC-25]} establishes a framework for the explicit full-discretization schemes, with a specific scheme given as an example.
In Section \ref{sec:moment-boundedness[weak4AC-25]} we prove the uniform moment bounds for the proposed full-discretizations. In Section \ref{sec:convergence[weak4AC-25]}, the uniform-in-time weak convergence is established, with explicit convergence rates revealed.
Numerical results are reported in Section \ref{sec:numerical[weak4AC-25]} to validate the theoretical findings.
The last section gives some concluding remarks.

\section{The considered SPDEs}\label{sec:SPDE_considered[weak4AC-25]}
 
Let $\N$ be a set of positive integers and let $\N_0 := \N \cup \{ 0\}$.
By $C$ we denote a positive constant that might change at different occurrences.
Sometimes, we write $C(a_1,a_2,...,a_n)$ to show the dependence of parameters $a_1,a_2,..., a_n$.
For $a,b \in \R$, we denote $a \wedge b := \min\{ a,b\}$ and $a \vee b := \max\{ a,b\}$.
For a bounded domain $\mathcal{D} \subset \R^d$, $d \leq 3$,
we let $ L^r ( \mathcal{D}; \R), r \geq 1$ ($L^r (\mathcal{D})$ or $L^r$ for short) denote the Banach space consisting of $r$-times integrable functions,
endowed with the usual norms $\| \cdot \|_{ L^r }$.
By $H := L^2(\mathcal{D}; \R)$ we denote a real separable Hilbert space, endowed with the usual product $\langle \cdot, \cdot \rangle$ and the norm $\| \cdot \| := \langle \cdot, \cdot \rangle^\frac{1}{2}$. 
Let $\mathcal{L}(H)$ denote
the space of bounded linear operators from $H$ to $H$, endowed with the usual operator norm $\| \cdot \|_{\mathcal{L}(H)}$. 
 By $\mathcal{L}_2(H) \subset \mathcal{L}(H)$ ($\mathcal{L}_2$ for short), we denote the subspace consisting of all Hilbert-Schmidt operators from $H$ to $H$, which is also a separable Hilbert space, endowed with the scalar product $\langle \Gamma_1, \Gamma_2 \rangle_{\mathcal{L}_2(H)} := \sum_{n \in \N} \langle \Gamma_1 \eta_n, \Gamma_2 \eta_n \rangle$ and the norm $\| \Gamma \|_{\mathcal{L}_2(H)} 
 :=
 \left( \sum_{n \in \N}
 \| \Gamma  \eta_n \|^2 \right)^\frac{1}{2}$, independent of the choice of the orthogonal basis $\{ \eta_n\}_{n \in \N}$ of $H$. 
The Banach space consisting of continuous functions will be denoted by $V := C(\mathcal{D}, \R)$,
endowed with the usual norm $\| \cdot \|_V$.
Let $\mathcal{C}_b^k(H), k \in \N$ denote
the space consisting of mappings from $H$ to $\R$, having bounded and uniformly continuous derivatives, up to the $k$-th order.

In this article we focus on the parabolic SPDE in the Hilbert space $H$:
\begin{equation}\label{eq:considered_SEE[weak4AC-25]}
\left\{\begin{array}{l}
\,\mathrm{d} X(t)  = - A X(t) \,\mathrm{d} t + F(X(t)) \,\mathrm{d} t + \,\mathrm{d} W(t), 
\quad t > 0, \\
X(0) = X_0.
\end{array}\right.
\end{equation}
Here the operators $A,F$, the noise process $W$ and the initial value $X_0$ satisfy the following assumptions.


\begin{assumption}[Linear Operator $A$]\label{assump:A(linear_operator)[weak4AC-25]}
Let $\mathcal{D} \subset \R^d$, $d \leq 3$, be a bounded spatial domain with smooth boundary.
Let $-A \colon Dom(A) \subset H \rightarrow H$ be the Laplacian on $\mathcal{D}$ with homogeneous Dirichlet boundary conditions, i.e., $-Au = \Delta u$, 
$u \in Dom(A) := H^2(\mathcal{D}) \cap H^1_0(\mathcal{D})  $.
\end{assumption}

Assumption \ref{assump:A(linear_operator)[weak4AC-25]} implies the existence of
an eigensystem
$\left\{\lambda_j, e_j\right\}_{j \in \mathbb{N}}$ in $H$ satisfying
$ A e_j = \lambda_j e_j$,
with
$\left\{ \lambda_j \right\}_{j \in \mathbb{N}}$ being an increasing sequence such that $\lim_{j \rightarrow \infty} \lambda_j = \infty$.
Moreover, $-A$ generates an analytic and contractive semi-group, denoted by 
$ E(t):=e^{-At},t \geq 0$.
By means of the spectral decomposition, we define the fractional powers of $A$, i.e., $A^\vartheta$ for $\vartheta \in \mathbb{R}$ \cite[Appendix B.2]{Kruse2014}. 
Denote the interpolation spaces by $\dot{H}^\vartheta := Dom( A^{\frac{\vartheta}{2}} ), \vartheta \in \mathbb{R}$, which are separable Hilbert spaces equipped with the inner product 
$\langle \cdot, \cdot \rangle_\vartheta := \langle A^{\frac{\vartheta}{2}} \cdot, A^{\frac{\vartheta}{2}} \cdot \rangle$ 
and the norm
$\| \cdot \|_\vartheta := \| A^{\frac{\vartheta}{2}} \cdot \| = \langle \cdot, \cdot \rangle_\vartheta^{1/2}$. 
The following regularity properties are well-known (see e.g. \cite{Pazy1983}):
for any $t>0, \vartheta \geq 0, \varsigma \in[0,1]$,
\begin{equation}\label{eq:E(t)_semigroup_property[weak4AC-25]}
    \left\|E(t) \right\|_{\mathscr{L}(H)} 
    \leq 
    e^{-\lambda_1 t}
,
\quad
    \left\| A^{\vartheta} E(t) \right\|_{\mathscr{L}(H)} 
    \leq 
    C t^{-\vartheta}
, 
\quad
    \left\| A^{-\varsigma} (I - E(t)) \right\|_{\mathscr{L}(H)} 
    \leq 
    C t^{\varsigma}
.
\end{equation}


\begin{assumption}[Noise Process]\label{assump:W(noise)[weak4AC-25]}
Let $ \{ W(t) \}_{ t \geq 0 } $ be an $H$-valued (possibly cylindrical) $Q$-Wiener process on the
stochastic basis $\left( \Omega, \mathcal{F}, \P, \left\{ \mathcal{F}_t \right\}_{ t \geq 0 } \right)$. Let the covariance operator $Q \in \mathcal{L}(H)$ be a bounded, self-adjoint,  positive semi-definite operator, satisfying that
\begin{equation}
 \big\| A^{ \frac{\gamma-1}{2} } 
         Q^{\frac{1}{2}}
\big\|_{ \mathscr{L}_2(H) }
< 
  \infty,
\quad
\text{ for some } \gamma \in (0 , 2 ].
\end{equation}
Moreover, assume that $A$ commutes with $Q$ in the case that $\gamma \leq \frac{d}{2}$.
\end{assumption}
%


\begin{assumption}[Nonlinearity]\label{assump:F(Nonlinearity)[weak4AC-25]}
Let $q \in [1, \tfrac{4 + 3d}{2d})$ be an integer for $d=1,2,3$
and let $F\colon L^{4 q-2}(\mathcal{D}) \rightarrow H$ be a nonlinear Nemytskii operator given by 
\begin{equation}
    F(u)(x) := f( u(x) ), 
    \quad
    x \in \mathcal{D},
\end{equation}
where $f (v) = -c_f v^{2q-1} + f_0(v), v \in \R$ with $c_f >0$ and $f_0 \colon \R \rightarrow \R$ being twice differentiable such that $|f_0(v)| \leq C( 1 + |v|^{2q-2}), v \in \R$. Moreover, there exist constants $L_f \in \R$ and $R_f, c_0, c_1, c_2, c_3, c_4 >0$ such that,
for all $u, v \in \R$,
\begin{align}
 f'(u) 
  & \leq
  L_f ,
\\
  |f'(u)|
  \vee
  |f''(u)|
  & \leq
  R_f ( 1 + |u|^{2q-2} ),
\\
\label{eq:asuume_(u+v)f(u)[weak4AC-25]}
   (  u + v ) 
    f( u )
   & \leq
    - c_0
     |u|^{2q}
    + c_1
     |v|^{2q}
    + c_2,
\\
\label{eq:asuume_|f(u)-f(v)|[weak4AC-25]}
  | f( u) - f( v ) |
   & \leq
   c_3
   ( |u|^{2q-2} + |v|^{2q-2} ) | u - v |
   + 
   c_4 
   | u - v |.
\end{align}
\end{assumption}
Obviously, setting $v=0$ in \eqref{eq:asuume_|f(u)-f(v)|[weak4AC-25]} leads to 
that for all $u \in \R$, 
\begin{equation}\label{eq:asuume_|f(u)|[weak4AC-25]}
    |f(u)| 
    \leq
      c_3 |u|^{2q-1} 
      +
      c_4 |u|
      +
      c_5 ,
\end{equation}
for some constant $c_5 := |f(0)| \geq 0$.
A typical example of the nonlinearity $f$ satisfying Assumption \ref{assump:F(Nonlinearity)[weak4AC-25]} is 
$f(u) := a_0 + a_1 u + a_2 u^2 + a_3 u^3$
    with $a_3 < 0, a_0,a_1,a_2 \in \R$.
Such SPDEs are called stochastic Allen–Cahn equations in the literature \cite{wang2020efficient,Cai2021weak4ACE,CUI2021weak,Brehier_Allen_C,QiWang2019optimal,HuangC_ShenJ2023Mathcomp}.
Moreover, for any 
$\upsilon, \psi, \psi_1, \psi_2 \in L^{4q-2}(\mathcal{D})$,
we define
\begin{equation}
    \big(
    F'(\upsilon)(\psi)
    \big)
    (x)
    :=
    f'( \upsilon(x)  )
    \psi(x),
    \quad
    \big(
    F''(\upsilon)(\psi_1, \psi_2)
    \big)
    (x)
    :=
    f''( \upsilon(x)  )
    \psi_1(x)
    \psi_2(x)
    ,
    \quad
    x \in \mathcal{D}.
\end{equation}

\begin{assumption}[Initial value]\label{assump:X_0(Initial Value)[weak4AC-25]}
    Let the initial value $ X_0\colon \Omega \rightarrow H $ be an $ \mathcal{F}_0 / \mathcal{B}(H) $-measurable random variable and let $ \gamma $ be determined by Assumption \ref{assump:W(noise)[weak4AC-25]}.
    For any $ p \geq 1 $ and for some
    $\varrho > \frac{ d }{2} $, there exists constant $ C(\gamma,\varrho,p)>0 $ depending on $ \gamma, \varrho, p $ such that
\begin{equation}
    \| X_0 \|_{ L^p( \Omega, \dot{H}^{\gamma \vee \varrho ) } } 
    \leq
     C(\gamma,\varrho, p) 
    < 
    \infty.   
\end{equation}
    
\end{assumption}

Under all the assumptions stated above, the well-posedness of the SPDE \eqref{eq:considered_SEE[weak4AC-25]} is established as follows (see, e.g., \cite[Chapter 6]{cerrai2001second} or \cite[Theorem 3.5]{wang2024IMA}).

\begin{thm}
    Suppose Assumptions \ref{assump:A(linear_operator)[weak4AC-25]}-\ref{assump:X_0(Initial Value)[weak4AC-25]} are satisfied. Then, the SPDE \eqref{eq:considered_SEE[weak4AC-25]} admits a unique mild solution $\{ X(t) \}_{t \geq 0}$ with continuous sample paths defined by
\begin{equation}\label{eq:spde[weak4AC-25]}
        X(t) = E(t) X_0
               +
               \int_0^t E(t-s) F( X(s) ) \,\mathrm{d}s 
               +
               \int_0^t E(t-s) \,\mathrm{d}W(s),
               \ 
               t \geq 0,
        \quad
        \mathbb{P} \text{-a.s.}.
    \end{equation}
\end{thm}

\section{The proposed explicit full-discretization schemes}\label{sec:scheme[weak4AC-25]}

To numerically solve the SPDE \eqref{eq:spde[weak4AC-25]}, we rely on the spectral Galerkin method for the spatial semi-discretization, based on which a class of explicit time-stepping schemes are introduced to get the space-time full discretization. 

\subsection{A general framework for full-discretization schemes}

We first approximate the underlying problem \eqref{eq:spde[weak4AC-25]} spatially, using
the spectral Galerkin method. For $N \in \mathbb N$, by spanning the $N$ first eigenvectors of the dominant linear operator $A$, we define a finite-dimensional linear subspace $H^N \subset H$ by
$ H^N := span \left\{ e_1, e_2, ..., e_N \right\}$, and the projection operator
$P_N \colon \dot H^\vartheta \rightarrow H^N$
by
$
    P_N \xi : = \sum_{i = 1}^N \langle \xi, e_i \rangle e_i, 
    \
    \xi \in \dot H^\vartheta , \vartheta \in \mathbb{R}.
$
It is easy to show
\begin{equation}\label{eq:P_N(regularity)[weak4AC-25]}
    \| (P_N - I) \xi \| 
    \leq 
    \lambda_{N+1}^{-\frac{\vartheta}{2}} 
    \| \xi \|_\vartheta , 
    \quad  
    \forall \ \xi \in \dot H^\vartheta , 
    \vartheta \geq 0.
\end{equation}
Then the spectral Galerkin approximation of (\ref{eq:considered_SEE[weak4AC-25]}) leads to the following finite-dimensional stochastic differential equations (SDEs) in $H^N$:
\begin{equation}\label{eq:spectral-25AC}
\left\{\begin{array}{l}
\,\mathrm{d} X^N(t) 
= 
 - A_N X^N(t) \,\mathrm{d} t 
 + P_N F(X^N(t)) \,\mathrm{d} t 
 + P_N \,\mathrm{d} W(t), 
 \quad t > 0, 
 \\
 X^N(0) = P_N X_0,
\end{array}\right.
\end{equation}
where we define
$A_N :=  A P_N $, generating an analytic semi-group in $H^N$, denoted by $E_N(t) := e^{-t A_N}, t \in [0, \infty)$. 
Then the unique solution of \eqref{eq:spectral-25AC} is given by
\begin{equation}\label{eq:spatial_disc[weak4AC-25]}
    X^N(t)
    =
    E_N(t) P_N X_0
    +
    \int_0^t E_N(t-r) P_N F( X^N(r) ) \,\mathrm{d}r
    +
    \int_0^t E_N(t-r) P_N \,\mathrm{d}W(r),
    \quad
    t \geq 0.
\end{equation}

Based on the spatial semi-discretization \eqref{eq:spatial_disc[weak4AC-25]}, 
we now propose an explicit full-discretization schemes as follows:
\begin{equation}\label{eq:full_discretization[weak4AC-25]}
    X^{N,\tau}_{t_{m+1}}
    = 
    E_N(\tau) X^{N,\tau}_{t_m} 
    + 
    \tau E_N(\tau)
    P_N F_{\tau,N}
    \big( X^{N,\tau}_{t_m} \big)
    +
    E_N(\tau) P_N \Delta W_m,
    \quad
    X_0^{N,\tau} = P_N X_0 =: X_0^N,
\end{equation}
where
we denote $t_m := m \tau$ for the uniform stepsize $\tau > 0 $, $\Delta W_m := W(t_{m+1}) - W(t_m)$, $ m \in \N_0$,
and
$F_{\tau,N}\colon L^{4q-2}(\mathcal{D}) \rightarrow H$
is given by 
\begin{equation}
    F_{\tau,N}(u)(x)
        : =
        f_{\tau,N}( u(x) ), 
        \quad
        x \in \mathcal{D}.
\end{equation}
Here
$f_{\tau,N}\colon \R \rightarrow \R$ is
a modification of the mapping $f$ satisfying Assumption \ref{assump:F_N,tau[weak4AC-25]} below.
By iteration, the full-discretization schemes \eqref{eq:full_discretization[weak4AC-25]} can be also rewritten as
\begin{align}\label{eq:full-dis_sum[weak4AC-25]}
     X^{N,\tau}_{t_{m}}
     =
      E_N( t_{m} ) X^{N,\tau}_{0}
      +
      \tau
      \sum_{k=0}^{m-1}
      E_N(t_{m} - t_k)
      P_N
      F_{\tau,N}
     \big( X^{N,\tau}_{t_k} \big)
     +
     \mathcal{O}^{N,\tau}_{t_{m}},
     \quad
     X^{N,\tau}_{0} = P_N X_0,
\end{align}
for $m \in \N_0$, where we denote 
the discretized version of stochastic convolution by
\begin{align}
   \label{eq:O^N,tau[weak4AC-25]}
    \mathcal{O}^{N,\tau}_{t_m} 
   & :=
         \sum_{k=0}^{m-1} E_N( t_{m} - t_{k} ) P_N \Delta W_k,
    \quad 
    m \in \N_0.
\end{align}
%
%

\begin{assumption}\label{assump:F_N,tau[weak4AC-25]}
    Let Assumption \ref{assump:A(linear_operator)[weak4AC-25]}, \ref{assump:F(Nonlinearity)[weak4AC-25]} be fulfilled.
%
Let $\lambda_N , N \in \N$ be the $N$-th eigenvalue of linear operator $A$.
Then there exists a constant $\tau^* \in (0, \infty)$ such that for $0< \tau \leq \tau^*$, 
    the modified function $f_{\tau,N}\colon \R \rightarrow \R$ satisfies the following conditions: for some $\theta > 0$ and $\rho >0$, $\alpha \in [0,1]$ satisfying $\alpha \rho < 1- \tfrac{d}{4} $,
    there exist constants $\widetilde{c}_0, \widetilde{c}_1, \widetilde{c}_2, \widetilde{c}_3,\widetilde{c}_4,
\widetilde{l}> 0$ independent of $\tau$ and $\lambda_N $, such that for any $u,v \in \R$,
    \begin{align}
    \label{eq:assume_(u+v)f_tau(u)[weak4AC-25]}
    2 
    ( u + v ) f_{\tau,N}(u) 
    +
    \tau
    | f_{\tau,N}(u) |^2
  &  \leq 
    - 2 \widetilde{c}_0 |u|^2 
    +
    2 \widetilde{c}_1 ( 1 + |v|^{2q } 
    ),  
    \\ 
    \label{eq:assume_|F_tau(u)|[weak4AC-25]}
    | f_{\tau,N}(u) |
   & \leq
    \widetilde{c}_2 
    |f(u)|,
    \\
\label{eq:assume_|F_tau(u)|_lambda[weak4AC-25]}
    | f_{\tau,N}(u) |
   & \leq
    \widetilde{c}_3
    ( 1 + |u|
      +
      \lambda_N^{ \alpha \rho }|u| )  ,
    \\
    \label{eq:assume_|F_tau(u)-F(u)|[weak4AC-25]}
    | f_{\tau,N}(u) - f(u) |
  &  \leq
    \widetilde{c}_4
    (\tau^{\theta} + \lambda_N^{-\rho})
     ( 1 + |u|^{ \widetilde{l} } ) |f(u)|.
    \end{align}
\end{assumption}

\begin{rem}
It is worthwhile to point out that, the last condition quantifies the discrepancy between the original $f$ and the modification $f_{\tau,N}$,
where the parameters $\theta$ and $\rho$ would put upper limits of convergence rates in time and space, respectively.
Moreover, the requirement $\alpha \rho < 1- \tfrac{d}{4} $ is essentially used to control
the linear growth constant of $f_{\tau,N}(u)$, i.e., $\lambda_N^{ \alpha \rho }$,
even in multiple dimensions $d =2, 3$ (see uniform bounds of $\mathcal{R}^{N,\tau}$ in the proof of Theorem \ref{thm:X^N,tau(bound)[weak4AC-25]}). Such a condition can be fulfilled by choosing a flexible degree $\alpha < (1- \tfrac{d}{4})/ \rho$ sufficiently small (see the concrete example \eqref{eq:example_f_tau[weak4AC-25]} below).
This
differs from existing tamed
schemes such as \cite{wang2020efficient,brehier2022ESAIM,wang2024IMA,liu-shen_2025geometric,Cai2021weak4ACE,HuangC_ShenJ2023Mathcomp} for SPDEs and \cite{angeli2023uniform,hutzenthaler2012strong,neufeld2025non,sabanis2016euler} for SODEs, where fixed degrees such as $\alpha = \tfrac12$ or $\alpha = 1$ were used. 

\end{rem}

In the next subsection, we give a concrete example of $f_{\tau,N}$ such that all conditions in Assumption \ref{assump:F_N,tau[weak4AC-25]} are satisfied.

\subsection{Concrete full-discretization schemes}

Now we present an example of $f_{\tau,N}$ fulfilling Assumption \ref{assump:F_N,tau[weak4AC-25]}. 
%
%

\vspace{2.5mm}

\noindent\textbf{A concrete example of $f_{\tau,N}$.}

\vspace{2.5mm}

We let $\alpha =0$ for the case $q=1$, and $\alpha \in (0,1]$ for $q > 1$. 
Let $f_{\tau,N}: \R \rightarrow \R$ be defined by
\begin{equation}\label{eq:example_f_tau[weak4AC-25]}
       f_{\tau,N}( u )
       :=
       \frac{ f( u ) }
       { \Big(  1 
            + 
            (\beta_1 \tau^{\theta} +  \beta_2 \lambda_N^{-\rho} )
          |u|^{ \frac{2q-2}{\alpha} } \Big)^\alpha
        },
        \quad
        u \in \R,
    \end{equation}
where the parameters
$
       \theta,
       \rho,
      \beta_1, \beta_2 >0$ and $\alpha \rho < 1- \tfrac{d}{4} $.
Here, $\lambda_N , N \in \N$ is the $N$-th eigenvalue of the linear operator $A$ and define $\tfrac{0}{0} :=1$.

We mention that when
$q=1$ (i.e., $f$ is globally Lipschitz), we take
$\alpha=0$ and 
the underlying scheme reduces to the standard exponential Euler scheme \cite{wang2014DCDSweak}. 
When $\alpha=\tfrac12$, $\beta_2=0$, the underlying scheme is similar to that proposed by \cite{neufeld2025non}.
For $\alpha=1$, $\beta_2=0$, the underlying scheme coincides with those in \cite{sabanis2016euler,angeli2023uniform}. Distinct from these works, here the conditions $\beta_1, \beta_2>0$ and $0 < \alpha \rho < 1- \tfrac{d}{4} $ will be essentially used later for the superlinear case $q>1$ and thus $\alpha$ should be small (less than $\tfrac12$) in multiple dimensions $d=2,3$ in order to achieve the uniform-in-time moment bounds and the desired weak convergence rates. 

\begin{prop}\label{prop:concrete_f_tau[weak4AC-25]}
        Let Assumption \ref{assump:A(linear_operator)[weak4AC-25]}, \ref{assump:F(Nonlinearity)[weak4AC-25]} hold.
        Let $N \in \N$ and $ 0< \tau \leq \tau^*$
        for some constant $\tau^* \in (0, \infty)$.
Moreover, we
assume
\begin{equation}\label{eq:assume_condition_beta1[weak4AC-25]}
    2(c_3+ \mathbbm{1}_{ \{q=1\} } c_4  )^2 \tau^{ 1-\theta\alpha } \leq c_0 \beta_1^{\alpha},
\end{equation}
where the parameters $\beta_1, \theta, \alpha$ come from \eqref{eq:example_f_tau[weak4AC-25]} and $c_0,c_3,c_4$ are from Assumption \ref{assump:F(Nonlinearity)[weak4AC-25]}.
Then $f_{\tau,N}$ determined by \eqref{eq:example_f_tau[weak4AC-25]} satisfies Assumption \ref{assump:F_N,tau[weak4AC-25]}.
        
\end{prop}

\begin{proof}
In what follows, we attempt to validate all conditions in Assumption \ref{assump:F_N,tau[weak4AC-25]} one by one.
\begin{itemize}
    \item Verification of \eqref{eq:assume_(u+v)f_tau(u)[weak4AC-25]}.
\end{itemize} 
In the case that $q=1$, we take $\alpha=0$ and thus $f_{\tau,N} = f$. By employing the assumptions \eqref{eq:asuume_(u+v)f(u)[weak4AC-25]}, \eqref{eq:asuume_|f(u)|[weak4AC-25]}, along with the condition \eqref{eq:assume_condition_beta1[weak4AC-25]} for $\alpha=0$, i.e.,
$2 (c_3 + c_4)^2 \tau \leq c_0 $, one gets
\begin{align}
      2( u + v ) f_{\tau,N}( u ) + \tau |f_{\tau,N}( u )|^2
& \leq
   -2 c_0 |u|^2 + 2 c_1 |v|^2 + 2c_2 + 2(c_3 + c_4)^2 \tau  |u|^2 + 2 c_5^2 \tau
\notag\\
& \leq
   - c_0 |u|^2 + 2 c_1 |v|^2 + 2c_2 + 2 c_5^2.
\end{align}

Next we focus on the case $q >1, \alpha \in (0,1]$. In view of \eqref{eq:asuume_(u+v)f(u)[weak4AC-25]} and \eqref{eq:asuume_|f(u)|[weak4AC-25]}, 
one obtains 
\begin{align}
  & 2( u + v ) f_{\tau,N}( u ) + \tau |f_{\tau,N}( u )|^2
\notag\\
   & =
   \frac{  2( u + v ) f( u ) }
   { (1 + (\beta_1 \tau^{\theta} +  \beta_2 \lambda_N^{-\rho} ) 
       |u|^{ \frac{2q-2}{\alpha} } )^\alpha
      } 
   +
   \frac{ \tau |f( u )|^2
        }{ (1 + (\beta_1 \tau^{\theta} +  \beta_2 \lambda_N^{-\rho} ) 
       |u|^{ \frac{2q-2}{\alpha} } )^{2\alpha}
      } 
\notag\\
   & \leq
   \frac{ 2 ( -c_0 |u|^{2q} + c_1 |v|^{2q} + c_2 ) }
   { (1 + (\beta_1 \tau^{\theta} +  \beta_2 \lambda_N^{-\rho} )  
       |u|^{ \frac{2q-2}{\alpha} } )^\alpha
      } 
   +
   \frac{ \tau 
         ( c_3|u|^{2q-1} + c_4|u| + c_5 )^2
        }{ (1 + (\beta_1 \tau^{\theta} +  \beta_2 \lambda_N^{-\rho} )  
       |u|^{ \frac{2q-2}{\alpha} } )^{2\alpha}
      } 
\notag\\
   & \leq
   2 c_1 |v|^{2q}
 +
   2 c_2
 +
   \frac{ -2c_0 |u|^{2q} 
         (1 + (\beta_1 \tau^{\theta} +  \beta_2 \lambda_N^{-\rho} ) 
       |u|^{ \frac{2q-2}{\alpha} } )^\alpha 
       +
          2^\alpha c_3^2 \tau 
            |u|^{4q-2}
         +
            \hat{c}_5 
    }
         { (1 + (\beta_1 \tau^{\theta} +  \beta_2 \lambda_N^{-\rho} )  
       |u|^{ \frac{2q-2}{\alpha} } )^{2\alpha}
      },
\end{align}
where we used the Young inequality that
$ ( c_3|u|^{2q-1} + c_4|u| + c_5 )^2 \leq 2^\alpha c_3^2 |u|^{4q-2} + \hat{c}_5  $ for $\alpha \in (0,1], q >1$ and some constant $\hat{c}_5 > 0$.
Before proceeding further, we claim
\begin{equation}\label{eq:ineq_for_(1+x)alpha[weak4AC-25]}
    (1+x)^\alpha \geq 2^{\alpha-1}(1+x^\alpha),
    \quad
    x \geq 0, \alpha \in (0,1].
\end{equation}
To validate this claim, it suffices to prove
$g(x) := (1+x)^\alpha - 2^{\alpha-1}(1+x^\alpha) \geq 0$ for $x \geq 0, \alpha \in [0,1]$.
It is easy to calculate $g'(x) =\alpha [(1 + x)^{\alpha-1}-2^{\alpha-1}x^{\alpha-1}]$ and show that $g'(x)\leq 0$ for $x \in [0,1]$, and $g'(x) \geq 0$ for $x \geq 1$. 
Since $g(1)=0$, we infer $g(x) \geq 0$ for $x \geq 0$.
Armed with \eqref{eq:ineq_for_(1+x)alpha[weak4AC-25]}, one deduces
\begin{align}
  & 2( u + v ) f_{\tau,N}( u ) + \tau |f_{\tau,N}( u )|^2
\notag\\
   & \leq
   2 c_1 |v|^{2q}
 +
   2 c_2
 +
   \frac{ -
          2^\alpha
          c_0 |u|^{2q} 
         (1 + (\beta_1 \tau^{\theta} +  \beta_2 \lambda_N^{-\rho} )^\alpha
         |u|^{(2q-2)} )
        +
          2^\alpha
          c_3^2 \tau 
          |u|^{4q-2}
        +
           \hat{c}_5 }
         { (1 + (\beta_1 \tau^{\theta} +  \beta_2 \lambda_N^{-\rho} )  
       |u|^{ \frac{2q-2}{\alpha} } )^{2\alpha}
      }
\notag\\
   & \leq
   2 c_1 |v|^{2q}
 +
   2 c_2 + \hat{c}_5
 +
   \frac{ -  2^\alpha c_0 |u|^{2q} 
          -  \tfrac12 \cdot 2^\alpha
             c_0 
             (\beta_1 \tau^{\theta} +  \beta_2 \lambda_N^{-\rho} )^{ \alpha } 
             |u|^{4q-2}
         }
   { ( 1 + (\beta_1 \tau^{\theta} +  \beta_2 \lambda_N^{-\rho} ) |u|^{ \frac{2q-2}{\alpha} } )^{ {2\alpha} } }
\notag\\
& \quad
+
   \frac{ 
          -  \tfrac{1}{2} \cdot 2^\alpha
             c_0 
             (\beta_1 \tau^{\theta} +  \beta_2 \lambda_N^{-\rho} )^{ \alpha } 
             |u|^{4q-2}
          +
           2^\alpha
          c_3^2 \tau 
          |u|^{4q-2}
         }
   { ( 1 + (\beta_1 \tau^{\theta} +  \beta_2 \lambda_N^{-\rho} ) |u|^{ \frac{2q-2}{\alpha} } )^{ {2\alpha} } }
\notag\\
   & \leq
   2 c_1 |v|^{2q}
 +
   2 c_2 + \hat{c}_5
 +
  2^\alpha 
   \bigg[
   \frac{ -  c_0 |u|^{2q-2} 
          -  \tfrac12 
             c_0 
             (\beta_1 \tau^{\theta} +  \beta_2 \lambda_N^{-\rho} )^{ \alpha } 
             |u|^{4q-4}
         }
   { ( 1 + (\beta_1 \tau^{\theta} +  \beta_2 \lambda_N^{-\rho} ) |u|^{ \frac{2q-2}{\alpha} } )^{ {2\alpha} } }
\Bigg]
\cdot
  |u|^2,
\label{eq:Ex2_(u+v)f_tau(u)+tau|f(u)|^2[weak4AC-25]}
\end{align}
where the condition \eqref{eq:assume_condition_beta1[weak4AC-25]} for $q>1$, i.e., $2 c_3^2 \tau \leq c_0 \beta_1^\alpha \tau^{\theta \alpha}$, was used in the last inequality.

To validate the assumption \eqref{eq:assume_(u+v)f_tau(u)[weak4AC-25]}, we consider two cases: $|u|\leq 1$ and $|u|>1$.
For the former case $|u| \leq 1$, one easily derives from \eqref{eq:Ex2_(u+v)f_tau(u)+tau|f(u)|^2[weak4AC-25]} that for any $\widetilde{c}_0 >0$,
\begin{align}\label{eq:|u|<1_verify[weak4AC-25]}
 2( u + v ) f_{\tau,N}( u ) + \tau |f_{\tau,N}( u )|^2
\leq
   2 c_1 |v|^{2q}
 +
   2 c_2 + \hat{c}_5
\leq
-2 \widetilde{c}_0 |u|^{2q}
+
   2 c_1 |v|^{2q}
 +
   2 c_2 + \hat{c}_5
+  
  2 \widetilde{c}_0,
\end{align}
as required.
For the other case $|u|>1$, we introduce an auxiliary function defined by
\begin{align}
\Upsilon(x)
    :=
     \frac{ - c_0 x 
            - \tfrac12 
             c_0 
             (\beta_1 \tau^{\theta} +  \beta_2 \lambda_N^{-\rho} )^{ \alpha } 
             x^2
         }
   { ( 1 + (\beta_1 \tau^{\theta} +  \beta_2 \lambda_N^{-\rho} ) x^{\frac{1}{\alpha}} )^{ {2\alpha} } },
   \quad
   x \geq 0, \alpha \in (0,1],
\end{align}
and therefore, the equation \eqref{eq:Ex2_(u+v)f_tau(u)+tau|f(u)|^2[weak4AC-25]} can be rewritten as 
\begin{align}
\label{eq:|u|>1_verify[weak4AC-25]}
   2( u + v ) f_{\tau,N}( u ) + \tau |f_{\tau,N}( u )|^2
\leq
2 c_1 |v|^{2q}
+ 
 2 c_2 + \hat{c}_5
 +
2^\alpha
\Upsilon(|u|^{2q-2}) 
\cdot  |u|^2.
\end{align}
In what follows, we attempt to prove
\begin{equation}\label{eq:Upsilon_bound[weak4AC-25]}
    \sup_{x \geq 1}
    \Upsilon(x)
    \leq
     \tfrac{ - c_0 }{ 2 
  (1 + \beta_1 (\tau^*)^\theta + \beta_2 
 \lambda_1^{-\rho}  )^{2\alpha} }
    =
    \tfrac{ -c_0 }{ 2 
  (1 + \beta_1 + \beta_2 \lambda_1^{-\rho} )^{2\alpha} },
\end{equation}
where, without loss of generality, one takes $\tau^* = 1$.
First, we compute
\begin{equation}\label{eq:Upsilon'(x)[weak4AC-25]}
\Upsilon'(x)
=
     \frac{ -  c_0 
            -  c_0
             (\beta_1 \tau^{\theta} +  \beta_2 \lambda_N^{-\rho} )^{ \alpha } 
             x
            +
             c_0 
             (\beta_1 \tau^{\theta} +  \beta_2 \lambda_N^{-\rho} )
             x^\frac{1}{\alpha}
         }
   { ( 1 + (\beta_1 \tau^{\theta} +  \beta_2 \lambda_N^{-\rho} ) x^{\frac{1}{\alpha}} )^{ 1+{2\alpha} } }
=
  \frac{  c_0 \cdot \Lambda (  (\beta_1 \tau^{\theta} +  \beta_2 \lambda_N^{-\rho} )^{ \alpha } 
             x )
         }
   { ( 1 + (\beta_1 \tau^{\theta} +  \beta_2 \lambda_N^{-\rho} ) x^{\frac{1}{\alpha}} )^{ 1+{2\alpha} } },
\end{equation}
where we further denote
\begin{equation}\label{eq:Tilde(Upsilon)'(x)_inEx2[weak4AC-25]}
    \Lambda(y)
    :=
    y^\frac{1}{\alpha} - y - 1,
    \quad
    y \geq 0.
\end{equation}
Since for $\alpha=1$, it is evident that $\Upsilon'(x) < 0$ and thus
\begin{align}
\sup_{x\geq 1}
    \Upsilon(x)
\leq
    \Upsilon(1)
=
   \tfrac{ - c_0  
            - \tfrac12 
             c_0 
             (\beta_1 \tau^{\theta} +  \beta_2 \lambda_N^{-\rho} )^{ \alpha } 
         }
   { ( 1 + \beta_1 \tau^{\theta} +  \beta_2 \lambda_N^{-\rho}  )^{ {2\alpha} } }
\leq
 \tfrac{ -c_0 }{  
  (1 + \beta_1 (\tau^*)^\theta + \beta_2 \lambda_1^{-\rho}
 )^{2\alpha} }
=
\tfrac{ -c_0 }{  
  (1 + \beta_1 + \beta_2 \lambda_1^{-\rho} )^{2\alpha} }.
\end{align}
For $\alpha \in (0,1)$, we deduce that
$
    \Lambda'(y) 
=
\tfrac{1}{\alpha} y^{ \frac{1}{\alpha} - 1 } -1 
\leq 0$ for $y \in [0,\alpha^{ \frac{\alpha}{1-\alpha} } ]$ and $ \Lambda'(y) \geq 0$ for $y \geq \alpha^{ \frac{\alpha}{1-\alpha} }$,
which means $\Lambda(y)$ is decreasing for $y \in [0,\alpha^{ \frac{\alpha}{1-\alpha} } ]$, while increasing for $y \geq \alpha^{ \frac{\alpha}{1-\alpha} }$.
Moreover, the facts $\Lambda(0) = -1$ and
$\Lambda(y) > 0$ for sufficiently large $y>0$
imply that
there exists a unique point $y^* > 0$ such that 
$\Lambda(y) \leq 0$ for $y \in[0, y^*]$
and
$\Lambda(y) \geq 0$ for $y \geq y^*$.
This yields
\begin{equation}
    \Upsilon'(x) \leq 0 \ \text{for}\ 
    x \in[0,x^*],
    \quad
    \Upsilon'(x) \geq 0 \ \text{for}\  
    x \geq x^*,
    \quad
    x^*:=
    \tfrac{y^*}{ 
    (\beta_1 \tau^{\theta} +  \beta_2 \lambda_N^{-\rho} )^{ \alpha }
    },
\end{equation}
in other words, $\Upsilon(x)$ is decreasing for $x \in[0,x^*]$, while increasing for $x \geq x^*$. Hence we infer that $\sup_{x\geq 1}\Upsilon(x)$ can be bounded by 
$\Upsilon(1)
    \vee
    \lim_{x \rightarrow \infty} 
\Upsilon(x)$, i.e.,
\begin{align}
    \Upsilon(x)
\leq
    \Upsilon(1)
    \vee
    \lim_{x \rightarrow \infty} 
\Upsilon(x)
=
   \tfrac{ - c_0  
            - \tfrac12 
             c_0 
             (\beta_1 \tau^{\theta} +  \beta_2 \lambda_N^{-\rho} )^{ \alpha } 
         }
   { ( 1 + \beta_1 \tau^{\theta} +  \beta_2 \lambda_N^{-\rho}  )^{ {2\alpha} } }
\vee
  \tfrac{-c_0}{
   2 (\beta_1 \tau^{\theta} +  \beta_2 \lambda_N^{-\rho} )^{ \alpha}
  } 
\leq
 \tfrac{ -c_0 }{ 2 
  (1 + \beta_1 + \beta_2 \lambda_1^{-\rho}
 )^{2\alpha} },
\end{align}
as required by \eqref{eq:Upsilon_bound[weak4AC-25]}.
Finally, by inserting this bound into the equation \eqref{eq:|u|>1_verify[weak4AC-25]} and recalling the estimate \eqref{eq:|u|<1_verify[weak4AC-25]},
the verification of the assumption \eqref{eq:assume_(u+v)f_tau(u)[weak4AC-25]} for the case $q>1, \alpha \in (0,1]$ is completed. 
\begin{itemize}
    \item Verification of
    \eqref{eq:assume_|F_tau(u)|[weak4AC-25]} and \eqref{eq:assume_|F_tau(u)|_lambda[weak4AC-25]}.
\end{itemize} 

The verification of
    \eqref{eq:assume_|F_tau(u)|[weak4AC-25]} is straightforward and thus is omitted.
To validate \eqref{eq:assume_|F_tau(u)|_lambda[weak4AC-25]},
we first notice that \eqref{eq:assume_|F_tau(u)|_lambda[weak4AC-25]} is obvious satisfied for the case that $q=1, \alpha=0$, due to the fact that $f_{\tau,N}=f$ for $\alpha=0$, along with
the assumption \eqref{eq:asuume_|f(u)|[weak4AC-25]}.
For the case that $q > 1, \alpha \in (0,1]$,
by employing the assumption \eqref{eq:asuume_|f(u)|[weak4AC-25]} and the inequality \eqref{eq:ineq_for_(1+x)alpha[weak4AC-25]},
we get for all $u \in \R$,
\begin{align}
    |f_{\tau,N}(u)|
 &\leq
    \frac{ c_3 |u|^{2q-1} + c_4 |u| + c_5 }
       { (  1 
            + (\beta_1 \tau^{\theta} +  \beta_2 \lambda_N^{-\rho} )
          |u|^{ \frac{2q-2}{\alpha} } )^\alpha
        }
\notag\\
 &\leq
    \frac{   c_3 |u|^{2q-1} + c_4 |u| + c_5 }
       { 2^{\alpha-1}(  1 
            + (\beta_1 \tau^{\theta} +  \beta_2 \lambda_N^{-\rho} )^\alpha
          |u|^{ 2q-2 } )
        }
\notag\\
 &\leq
   2^{1-\alpha} \beta_2^{-\alpha}
   c_3 
   \lambda_N^{\alpha \rho} |u|
   +
   2^{1-\alpha}(c_4|u| + c_5). 
\end{align}
This confirms the assumption \eqref{eq:assume_|F_tau(u)|_lambda[weak4AC-25]}.

\begin{itemize}
    \item Verification of \eqref{eq:assume_|F_tau(u)-F(u)|[weak4AC-25]}.
\end{itemize} 

First, we introduce an auxiliary function
\begin{equation}
    \Theta(x) := ( 1 + (\beta_1 \tau^{\theta} +  \beta_2 \lambda_N^{-\rho} ) x )^{-\alpha}, 
    \quad
    x \geq 0.
\end{equation}
Using this notation, $f_{\tau,N}$ can be rewritten as
$f_{\tau,N}(u) = f(u)\Theta(u^{ \frac{2q-2}{\alpha} }),u \in \R  $. Noting that
\begin{align}
| \Theta'(x) |
    = 
    \left|
    \tfrac{\alpha (\beta_1 \tau^{\theta} +  \beta_2 \lambda_N^{-\rho} ) }{ ( 1 + (\beta_1 \tau^{\theta} +  \beta_2 \lambda_N^{-\rho} ) x )^{\alpha+1}  }
    \right|
\leq
  \alpha (\beta_1 \tau^{\theta} +  \beta_2 \lambda_N^{-\rho} )
,
\quad
x \geq 0,
\end{align}
and $\Theta(0)=1$, we see for all $u \in \R$,
\begin{align}
        | f_{\tau,N}(u) - f(u) |
    =
      |f(u)|
      \cdot
      \big|\Theta( u^{\frac{2q-2}{\alpha}} ) - \Theta(0)\big|
    \leq
      \alpha (\beta_1 \tau^{\theta} +  \beta_2 \lambda_N^{-\rho} )
      u^{\frac{2q-2}{\alpha}} |f(u)|,
\end{align}
validating the
assumption \eqref{eq:assume_|F_tau(u)-F(u)|[weak4AC-25]}.
\end{proof}

\section{Uniform moment bounds of the fully discrete schemes}
\label{sec:moment-boundedness[weak4AC-25]}

In this section, we aim to derive uniform-in-time moment bounds of the space-time full-discretization schemes.
To this end, we first
%
%
show the uniform moment bounds for the discretized version of the stochastic convolution $\mathcal{O}^{N,\tau}_{t_m} ,m \in \N_0$ that have been extensively studied, with the aid of the Sobolev embedding inequality, see e.g., \cite[Lemma 4.1]{chen2020fullANM} and \cite[Lemma 3.5]{wang2014DCDSweak}.

\begin{lem}\label{lem:O^N,tau(bound)[weak4AC-25]}
    Let Assumptions \ref{assump:A(linear_operator)[weak4AC-25]}, \ref{assump:W(noise)[weak4AC-25]} hold. 
    Let $\mathcal{O}^{N,\tau}_{t_m} ,m \in \N_0$ be the discretized stochastic convolution defined by \eqref{eq:O^N,tau[weak4AC-25]}.
    Then for any $p \geq 1$, there exists a constant $ C(Q,p) > 0 $ such that 
    \begin{equation}
        \sup_{m \in \N_0} 
        \big\|
        \mathcal{O}^{N,\tau}_{t_m} 
        \big\|_{L^p( \Omega; V )}
        +
        \sup_{m \in \N_0} 
        \big\| 
        \mathcal{O}^{N,\tau}_{t_m} 
        \big\|_{L^p( \Omega; \dot{H}^{\gamma} )}
        \leq
        C( Q,p )
        <
        \infty.
    \end{equation}
\end{lem}

In addition, we establish a contractive property of the semi-group operator $E(t), t\geq 0$ in $L^{4q-2}$.

\begin{prop}
[Contractive property of the semi-group]
\label{prop:E_N(t)_contractivity[weak4AC-25]}
    Let the linear operator $A$ satisfy Assumption \ref{assump:A(linear_operator)[weak4AC-25]}.
    For all $u_0 \in L^{4q-2}$, for any $ t \geq 0 $ and integer $q \geq 1$, the semi-group operator $ E(t):=e^{-At},t \geq 0$ satisfies
    \begin{equation}
        \| E(t) u_0 \|_{L^{4q-2}} \leq \| u_0 \|_{ L^{4q-2} }.
    \end{equation}
\end{prop}
\begin{proof}
Since the assertion is trivial for $t=0$,  it suffice to show the assertion for the case $t >0$.
Let us look at the following linear problem 
\begin{equation}
    \frac{ \partial u(t) }{ \partial t }
    =
     - A u(t),
     \
     t >0,
    \quad
     u(0)=u_0,
\end{equation}
whose unique solution can be written as $ u(t) := E(t) u_0, t \geq 0 $.
Applying integration by parts,
one deduces that, for $t >0$,
\begin{equation}
    \begin{aligned}
        \frac{ 
        \partial 
        \| u(t)\|^{4q-2}
        _{ L^{4q-2} }
        }
        { \partial t }
   &=
        \int_{\mathcal{D}}
        (4q-2)
        ( u(t)(x) )^{4q-3}
        \frac{ 
        \partial  u(t)(x)
        }
        { \partial t }
        \,\mathrm{d}x   
\\
   &=
        - (4q-2) \int_{\mathcal{D}}
        ( u(t)(x) )^{4q-3}
        \cdot
        A u(t)(x)
        \,\mathrm{d}x         
\\
   &=
       - (4q-2)
       \langle
        ( u(t) )^{4q-3}
        ,
         A u(t)
       \rangle   
\\
   &=
       - (4q-2)(4q-3)
       \langle
        ( u(t) )^{4q-4}
        \nabla( u(t) )
        ,
        \nabla( u(t) )
       \rangle 
\\
    &\leq 0,
    \end{aligned}   
\end{equation}
which implies the contractive property $ \| E(t) u_0 \|_{L^{4q-2}} = \| u(t) \|_{L^{4q-2}} \leq \| u(0) \|_{L^{4q-2}}
=
\| u_0 \|_{L^{4q-2}}$.
\end{proof}
%
    We highlight that, by offering a contractive property of $E(t)$ in $L^{4q-2}$-norm ($q \geq 1$),  Proposition \ref{prop:E_N(t)_contractivity[weak4AC-25]} plays a vital role in establishing the uniform moment bounds of the full-discretization schemes.
In existing works on long-time error analysis, researchers usually relied on the standard estimates
$\|E(\tau) u_0 \|_{L^p} \leq C_p e^{-\tau } \|u_0\|_{L^p}, \tau \geq 0$ for $p > 2, C_p > 0$ (see e.g., \cite[(2.4)]{brehier2022ESAIM}),
which is not enough for us to prove uniform-in-time moment bounds of exponential schemes as $C_p e^{-\tau} > 1$ for small $\tau$.



Also, we need the following inequality.
\begin{lem}\label{lem:ineq_a+taub[weak4AC-25]}
    Let $q \geq 1$ be any integer, $\upsilon > 0, \tau >0$. Then for all $\mathbb{A} ,\mathbb{B} \geq 0$, it holds that
\begin{equation}
    (\mathbb{A} + \tau \mathbb{B})^{2q-1}
    \leq
    e^{ (2q-2) \upsilon\tau} \mathbb{A}^{2q-1}  
    + 
     \tau
\left( 
  \tau^{2q-2} 
  + 
  (1 +  ( \tfrac{2}{ \upsilon } )^{2q-1} )
  e^{ (2q-2)\tau }
 \right)  
   \mathbb{B}^{2q-1}.
\end{equation}
\end{lem}
\begin{proof}
For $\mathbb{A}, \mathbb{B} \geq 0$, $q \geq 1$ and $\upsilon > 0$,
we utilize the Young inequality
$xy \leq \epsilon \tfrac{x^\vartheta}{\vartheta} + \epsilon^{-\frac{\varsigma}{\vartheta}} \tfrac{y^\varsigma}{\varsigma} $ 
for 
$x = \mathbb{A}^{2q-1-j}, y= \mathbb{B}^{j}, \vartheta = \tfrac{2q-1}{2q-1-j}, \varsigma=\tfrac{2q-1}{j}$ and $\epsilon = (\tfrac{\upsilon }{2} )^j$ to arrive at
\begin{equation}
\begin{aligned}
 (\mathbb{A} + \tau \mathbb{B})^{2q-1} 
& =
   \mathbb{A}^{2q-1}
   +
   \tau^{2q-1}
   \mathbb{B}^{2q-1}
   +
   \sum_{j=1}^{2q-2}
     \tfrac{(2q-1)!}{j!(2q-1-j)!}
     \tau^j
     \mathbb{A}^{2q-1-j}
     \cdot
     \mathbb{B}^{j}
\\ 
& \leq 
   \mathbb{A}^{2q-1}
   +
   \tau^{2q-1}
   \mathbb{B}^{2q-1}
   +
   \sum_{j=1}^{2q-2}
\Big(
     \tfrac{(2q-1)!}{j!(2q-1-j)!}
     \tau^j
      (\tfrac{ \upsilon}{2} )^j
      \tfrac{2q-1-j}{2q-1}
 \Big)
     \mathbb{A}^{2q-1}
\\
&\quad\quad
 +
   \sum_{j=1}^{2q-2}
   \Big(
       \tfrac{(2q-1)!}{j!(2q-1-j)!}
       \tau^j 
    ( \tfrac{2}{ \upsilon } )^{2q-1-j}
    \tfrac{j}{2q-1}
   \Big)
    \mathbb{B}^{2q-1}.
\end{aligned}
\end{equation}
By deducing that
\begin{align}
\sum_{j=1}^{2q-2}
    \tfrac{(2q-1)!}{j!(2q-1-j)!}
    \tau^j
     (\tfrac{ \upsilon }{2} )^j
    \tfrac{2q-1-j}{2q-1}
\leq
\sum_{j=1}^{2q-2}
    \tfrac{(2q-2)^j }{j!}
    \tau^j
     (\tfrac{ \upsilon }{2} )^j
\leq
(q-1) \upsilon \tau
\sum_{j=1}^{2q-2}
    \tfrac{ ((q-1)\upsilon \tau  )^{j-1}  }{j!}
\leq
   \tau
   (q-1) \upsilon 
    e^{(q-1) \upsilon \tau },
\end{align}
and similarly
\begin{align}
    \sum_{j=1}^{2q-2}
       \tfrac{(2q-1)!}{j!(2q-1-j)!}
       \tau^j 
    ( \tfrac{2}{ \upsilon } )^{2q-1-j}
    \tfrac{j}{2q-1}
\leq
    \sum_{j=1}^{2q-2}
       \tfrac{(2q-2)^{j-1} }{(j-1)!}
       \tau^j 
    ( \tfrac{2}{ \upsilon } )^{2q-1-j}
\leq
  \tau
(1 +  ( \tfrac{2}{ \upsilon } )^{2q-1} )
  e^{ (2q-2)\tau },
\end{align}
we show
\begin{equation}
\begin{aligned}
 (\mathbb{A} + \tau \mathbb{B})^{2q-1} 
& \leq 
   \Big( 1 
       +
       (q-1) \upsilon 
       \tau
       e^{(q-1) \upsilon \tau } 
    \Big)
   \mathbb{A}^{2q-1}
   +
   \tau^{2q-1}
   \mathbb{B}^{2q-1}
  + 
  \tau
(1 +  ( \tfrac{2}{ \upsilon } )^{2q-1} )
  e^{ (2q-2)\tau }
   \mathbb{B}^{2q-1}
\\
&
\leq
(1+ (q-1) \upsilon 
       \tau)
e^{(q-1) \upsilon  \tau }
   \mathbb{A}^{2q-1}
   +
 \tau
\left( 
  \tau^{2q-2} 
  + 
  (1 +  ( \tfrac{2}{ \upsilon } )^{2q-1} )
  e^{ (2q-2)\tau }
 \right)  
   \mathbb{B}^{2q-1}
\\
&
\leq
e^{2(q-1) \upsilon  \tau }
   \mathbb{A}^{2q-1}
   +
 \tau
\left( 
  \tau^{2q-2} 
  + 
  (1 +  ( \tfrac{2}{ \upsilon } )^{2q-1} )
  e^{ (2q-2)\tau }
 \right)  
   \mathbb{B}^{2q-1},
\end{aligned}
\end{equation}
and thus obtain the desired result.
\end{proof}

In the forthcoming lemma, we first prove the uniform moment bounds of $X^{N,\tau}$ in $L^2$-norm.

\begin{lem}[Uniform moment bounds in $L^2$-norm]
\label{lem:X^N,tau(L2_bound)[weak4AC-25]}

    Let Assumptions \ref{assump:A(linear_operator)[weak4AC-25]}-\ref{assump:X_0(Initial Value)[weak4AC-25]}, \ref{assump:F_N,tau[weak4AC-25]} hold 
    and let $ X^{N,\tau}_{t_m}$ be defined by \eqref{eq:full_discretization[weak4AC-25]}.
    For any $p \geq
    1$, there exists a constant
    $C( X_0, Q, p, q ) > 0 $ such that,
    \begin{equation}
        \sup_{m \in \N_0}
        \big\|
          X^{N,\tau}_{t_m} 
        \big\|_{L^p( \Omega; H )}
    \leq
        C( X_0, Q, p, q )
    < \infty.
    \end{equation} 
    
\end{lem}

\begin{proof}
 Recalling the full-discretization schemes \eqref{eq:full-dis_sum[weak4AC-25]}, we introduce
\begin{align}\label{eq:Y^N,tau_sum[weak4AC-25]}
    Y^{N,\tau}_{t_{m}}
    :=
    X^{N,\tau}_{t_{m}}
    -
    \mathcal{O}^{N,\tau}_{t_{m}}
=
    E_N( t_{m} ) Y^{N,\tau}_{0}
   +
    \tau
      \sum_{k=0}^{m-1}
      E_N(t_{m} - t_k)
      P_N
      F_{\tau,N}
     \big( Y^{N,\tau}_{t_k} + \mathcal{O}^{N,\tau}_{t_k} \big)
\end{align}
for $m \in \N_0$.
It is clear that
\begin{align}\label{eq:Y^N,tau[weak4AC-25]}
    Y^{N,\tau}_{t_{m+1}}
    =
      E_N(\tau)  
      \left(   
     Y^{N,\tau}_{t_m} 
      +
     \tau 
     P_N
    F_{\tau,N}( Y^{N,\tau}_{t_m}  +
           \mathcal{O}^{N,\tau}_{t_m} ) 
      \right),
    \ 
    m \in \N_0,
    \quad
    Y^{N,\tau}_{0} = X^{N,\tau}_{0}.
\end{align}
With the aid of the decomposition 
that
$X^{N,\tau} = Y^{N,\tau} + \mathcal{O}^{N,\tau}$
and Lemma \ref{lem:O^N,tau(bound)[weak4AC-25]},
it suffices to bound for $Y^{N,\tau}$.
By using \eqref{eq:E(t)_semigroup_property[weak4AC-25]}, one derives from \eqref{eq:Y^N,tau[weak4AC-25]} that for all $m \in \N_0$,
\begin{equation}
\begin{aligned}
\big\| 
Y^{N,\tau}_{t_{m+1}} 
\big\|^2
&  = 
\big\| E_N(\tau)  
     (   
     Y^{N,\tau}_{t_m} 
      +
     \tau 
     P_N
    F_{\tau,N}( Y^{N,\tau}_{t_m}  +
           \mathcal{O}^{N,\tau}_{t_m} ) 
      )
\big\|^2
\\
&\leq
\big\|   
     Y^{N,\tau}_{t_m} 
      +
     \tau 
    F_{\tau,N}( Y^{N,\tau}_{t_m}  +
           \mathcal{O}^{N,\tau}_{t_m} ) 
\big\|^2
\\
&\leq
\big\|
Y^{N,\tau}_{t_m}
\big\|^2
+
2 \tau
\big\langle
 Y^{N,\tau}_{t_m}
 ,
 F_{\tau,N}( Y^{N,\tau}_{t_m}  +
           \mathcal{O}^{N,\tau}_{t_m} ) 
\big\rangle
+
\tau^2
\big\|
 F_{\tau,N}( Y^{N,\tau}_{t_m}  +
           \mathcal{O}^{N,\tau}_{t_m} ) 
\big\|^2
\\
&\leq
\big\|
Y^{N,\tau}_{t_m}
\big\|^2
-
2 \widetilde{c}_0 \tau
\big\|  Y^{N,\tau}_{t_m}
     +
     \mathcal{O}^{N,\tau}_{t_m}
\big\|^2
+
2 \widetilde{c}_1 \tau
\big( 1 + \|\mathcal{O}^{N,\tau}_{t_m}\|^{2q}_V 
\big),
\end{aligned}
\end{equation}
where the assumption \eqref{eq:assume_(u+v)f_tau(u)[weak4AC-25]} was used in the last inequality. The Young inequality that
$4 \langle Y^{N,\tau}_{t_m} , \mathcal{O}^{N,\tau}_{t_m} \rangle 
\leq
\| Y^{N,\tau}_{t_m}\|^2 + 4 \|\mathcal{O}^{N,\tau}_{t_m}\|^2
$
further implies
\begin{equation}
\begin{aligned}
    \big\| 
   Y^{N,\tau}_{t_{m+1}} 
  \big\|^2
&\leq
   \big\| 
   Y^{N,\tau}_{t_m} 
  \big\|^2
 -2
\widetilde{c}_0 \tau
\big\|  
  Y^{N,\tau}_{t_m}
\big\|^2
  + 4
\widetilde{c}_0 \tau
\big\|  
  Y^{N,\tau}_{t_m}
\big\|
\big\|  
  \mathcal{O}^{N,\tau}_{t_m}
\big\|
  -2
\widetilde{c}_0 \tau
\big\|  
 \mathcal{O}^{N,\tau}_{t_m}
\big\|^2
+
2 \widetilde{c}_1 \tau
\big( 1 + \|\mathcal{O}^{N,\tau}_{t_m}\|^{2q}_V 
\big)
\\
&\leq
   \big\| 
   Y^{N,\tau}_{t_m} 
  \big\|^2
 -
\widetilde{c}_0 \tau
\big\|  
  Y^{N,\tau}_{t_m}
\big\|^2
+2
\widetilde{c}_0 \tau
\big\|  
 \mathcal{O}^{N,\tau}_{t_m}
\big\|^2
+
2 \widetilde{c}_1 \tau
\big( 1 + \|\mathcal{O}^{N,\tau}_{t_m}\|^{2q}_V 
\big)
\\
&\leq
   e^{-\widetilde{c}_0 \tau}
   \big\| 
   Y^{N,\tau}_{t_m} 
  \big\|^2
+
2 (\widetilde{c}_0 + \widetilde{c}_1) \tau
\big( 1 + \|\mathcal{O}^{N,\tau}_{t_m}\|^{2q}_V 
\big)
\\
&\leq
   e^{-(m+1)\widetilde{c}_0 \tau}
   \big\| 
   Y^{N,\tau}_{0} 
  \big\|^2
+
2 (\widetilde{c}_0 + \widetilde{c}_1) \tau
\sum_{k=0}^{m}
 e^{-(m-k)\widetilde{c}_0 \tau}
\big( 1 + \|\mathcal{O}^{N,\tau}_{t_k}\|^{2q}_V 
\big),
\end{aligned}
\end{equation}
where we used the estimate that $x < e^{x-1}$ for $x \in (0,1)$ in the third inequality.
Noting that $\sup_{m \in \N_0}\tau
    \sum_{k=0}^{m}
    e^{ -(m-k) \widetilde{c}_0 \tau } < \infty$,
by further employing
Assumption \ref{assump:X_0(Initial Value)[weak4AC-25]}
and
Lemma \ref{lem:O^N,tau(bound)[weak4AC-25]}, we thus arrive at the desired result.
\end{proof}

We emphasize that the uniform moment bounds in $L^2$-norm are insufficient to carry out the weak error analysis, where uniform moment bounds in $V$-norm are required.
However, it is highly non-trivial to show the uniform moment bounds in $V$-norm, as established in the following theorem.

\begin{thm}[Uniform moment bounds in $V$-norm]
\label{thm:X^N,tau(bound)[weak4AC-25]}
    Let Assumptions \ref{assump:A(linear_operator)[weak4AC-25]}-\ref{assump:X_0(Initial Value)[weak4AC-25]} and Assumption \ref{assump:F_N,tau[weak4AC-25]} hold. Let $ X^{N,\tau}_{t_m}, m \in \N_0$ be defined by \eqref{eq:full_discretization[weak4AC-25]}. For any $p \geq 1$, $\kappa \in [0, \gamma)$, there exists a constant $  C( X_0, Q, p, q, \kappa, \alpha, \theta, \rho ) >0 $ such that
    \begin{equation}
        \sup_{m \in \N_0} 
        \big\|
          X^{N,\tau}_{t_m} 
        \big\|_{L^p( \Omega; V )}
        +
        \sup_{m \in \N_0} 
        \big\| 
          X^{N,\tau}_{t_m} 
        \big\|_{L^p( \Omega; \dot{H}^{\kappa} )}
        \leq
       C( X_0, Q, p, q,\kappa, \alpha, \theta, \rho)
       < 
       \infty.
    \end{equation}   
\end{thm}    

\begin{proof}
Next we complete the proof in four steps.

\textbf{Step 1: Decomposition of $X^{N,\tau}$.}
\vspace{1.5mm}

By introducing the processes
\begin{equation}
    \mathcal{R}_{t_{m}}^{N,\tau}
    :=
  \tau
     \sum_{k=0}^{m-1} 
         E( t_{m} - t_{k} ) (P_N - I)
         F_{\tau,N}( X^{N,\tau}_{t_k} )
       +
       \mathcal{O}_{t_{m}}^{N,\tau}
       ,
       \quad
        \widetilde{Y}^{N,\tau}_{t_{m}}
    :=
    X^{N,\tau}_{t_{m}} - \mathcal{R}_{t_{m}}^{N,\tau},
    \quad
    m \in \N_0,
\end{equation}
and recalling \eqref{eq:full-dis_sum[weak4AC-25]}, we make a decomposition as follows:
\begin{equation}\label{eq:X_decompo_Y_R[weak4AC-25]}
     X^{N,\tau}_{t_{m}}
     =
     \widetilde{Y}^{N,\tau}_{t_{m}}
     +
     \mathcal{R}_{t_{m}}^{N,\tau},
     \quad
     m \in \N_0.
\end{equation}
Clearly, we get for all $ m \in \N_0$,
\begin{equation}
    \widetilde{Y}^{N,\tau}_{t_{m}}
        =
        E_N(t_m) \widetilde{Y}^{N,\tau}_{0}
        +
         \tau
       \sum_{k=0}^{m-1} 
         E( t_{m} - t_{k} )
         F_{\tau,N}( X^{N,\tau}_{t_k} )
    =
            E_N(t_m) \widetilde{Y}^{N,\tau}_{0}
            +
             \tau
       \sum_{k=0}^{m-1} 
         E( t_{m} - t_{k} )
         F_{\tau,N}( \widetilde{Y}^{N,\tau}_{t_k} +
           \mathcal{R}^{N,\tau}_{t_k} ).
\end{equation}
Evidently,
one has 
\begin{equation}\label{eq:widetilde(Y)^N,tau[weak4AC-25]}
        \widetilde{Y}^{N,\tau}_{t_{m+1}} 
   = 
      E(\tau)  
      \left(   
          \widetilde{Y}^{N,\tau}_{t_m} 
          +
          \tau 
           F_{\tau,N}( \widetilde{Y}^{N,\tau}_{t_m} +
           \mathcal{R}^{N,\tau}_{t_m} ) 
      \right),
      \ 
      m \in \N_0,
      \quad
       \widetilde{Y}^{N,\tau}_{0} = X^{N,\tau}_{0}.
\end{equation}

\vspace{1.5mm}
\textbf{Step 2: Uniform bounds of
$
   \mathcal{R}^{N,\tau}$.}
\vspace{1.5mm}

By recalling $1 - \alpha \rho > \tfrac{d}{4} $ in Assumption \ref{assump:F_N,tau[weak4AC-25]} and employing the Sobolev embedding theorem, we obtain for $1 -\alpha \rho - \delta> \tfrac{d}{4}$ with some $\delta \in (0, \tfrac{1 -\alpha \rho -\frac{d}{4} }{2}) $ that
\begin{align}\label{eq:Vbound_mathcal(R)[weak4AC-25]}
&   \sup_{m \in \N_0}
   \left\| 
          \mathcal{R}^{N,\tau}_{t_{m+1}} 
    \right\|_{L^p( \Omega; V )}
\leq 
 \sup_{m \in \N_0}
\left\|
    \tau
    \sum_{k=0}^m 
    E_N( t_{m+1} - t_{k} ) 
    ( P_N - I )
    F_{\tau,N}( X^{N,\tau}_{t_k} ) 
   \right\|_{L^p( \Omega; V )}
   +
 \sup_{m \in \N_0}
   \left\|
   \mathcal{O}^{N,\tau}_{t_{m+1}}
    \right\|_{L^p( \Omega; V )}
\notag    \\
& \leq
C
 \sup_{m \in \N_0}
   \left\|
   \tau
    \sum_{k=0}^m 
    A^{1 -\alpha \rho - \delta}
    E_N( t_{m+1} - t_{k} ) 
    ( P_N - I )
    F_{\tau,N}( X^{N,\tau}_{t_k} ) 
   \right\|_{L^p( \Omega; H )}
   +
   C(Q,p)
\notag    \\
& \leq
C
\tau
 \sup_{m \in \N_0}
   \sum_{k=0}^m 
\left\| A^{-\alpha \rho }( P_N - I ) \right\|_{\mathcal{L}(H)}
   \left\|
    A^{1 - \delta}
    E_N( t_{m+1} - t_{k} ) 
   \right\|_{\mathcal{L}(H)}
   \left( 1 +
     (1 + \lambda_N^{\alpha \rho})
    \big\| X^{N,\tau}_{t_k} \big\|_{L^p( \Omega; H )}
    \right)
       +
   C(Q,p)
\notag    \\
& \leq
C \tau
 \sup_{m \in \N_0}
   \sum_{k=0}^m 
   e^{-\frac12 \lambda_1 ( t_{m+1} - t_{k} ) }
   ( t_{m+1} - t_{k} )^{-1+\delta}
\left( 1 +
    \big\| X^{N,\tau}_{t_k} \big\|_{L^p( \Omega; H )} 
\right)
       +
   C(Q,p)
\notag \\
& < 
\infty,
\end{align}
where we utilized Lemma \ref{lem:O^N,tau(bound)[weak4AC-25]} in the second inequality and the assumption \eqref{eq:assume_|F_tau(u)|_lambda[weak4AC-25]} in the third inequality.
The estimate
$\| A^{1-\delta} E(t) \|_{\mathcal{L}(H)} 
\leq
e^{-\frac{\lambda_1 t}{2}}  \| A^{1-\delta} E(\frac{t}{2}) \|_{\mathcal{L}(H)}
, t \geq 0$ and the property \eqref{eq:E(t)_semigroup_property[weak4AC-25]} were also employed in the fourth inequality, while the last inequality holds true due to the fact that
$\tau
 \sup_{m \in \N_0}
   \sum_{k=0}^m 
   e^{-\frac12 \lambda_1 ( t_{m+1} - t_{k} ) }
   ( t_{m+1} - t_{k} )^{-1+\delta}< \infty$.

\vspace{1.5mm}
\textbf{Step 3: Uniform bounds for
$      
          \widetilde{Y}^{N,\tau}.
       $}
\vspace{1.5mm}

Firstly, we deduce the uniform bounds for
$      
          \widetilde{Y}^{N,\tau}
       $ in $L^{4q-2}$-norm.
In view of Proposition \ref{prop:E_N(t)_contractivity[weak4AC-25]}, for all $m \in \N_0$, one derives from \eqref{eq:widetilde(Y)^N,tau[weak4AC-25]} that
\begin{equation}\label{eq::Y^N,tau(L^4q-2_bound)first_step_inprf[weak4AC-25]}
\begin{aligned}
    \big\|
     \widetilde{Y}^{N,\tau}_{t_{m+1}} 
    \big\|_{ L^{4q-2} }
&=    
    \big\|
    E(\tau)
    \big( 
        \widetilde{Y}^{N,\tau}_{t_m} 
       +
       \tau
       F_{\tau,N}
       (  X^{N,\tau}_{t_m} ) 
    \big)
    \big\|_{ L^{4q-2} }   
\leq    
    \big\|
      \widetilde{Y}^{N,\tau}_{t_m}
     +
     \tau F_{\tau,N}(   \widetilde{Y}^{N,\tau}_{t_m}  +\mathcal{R}^{N,\tau}_{t_m} )  
    \big\|_{ L^{4q-2} } . 
\end{aligned}
\end{equation}
By similar arguments as in Lemma \ref{lem:X^N,tau(L2_bound)[weak4AC-25]}, one sees
\begin{align}
 &
 \big| 
      \widetilde{Y}^{N,\tau}_{t_m}(\cdot) 
     + 
     \tau f_{\tau,N}
     \big(   \widetilde{Y}^{N,\tau}_{t_m}(\cdot) +\mathcal{R}^{N,\tau}_{t_m}(\cdot) 
     \big)
     \big|^2
  \notag \\
&
= 
     \big|  \widetilde{Y}^{N,\tau}_{t_m}(\cdot) \big|^2
     +
     2 \tau
      \widetilde{Y}^{N,\tau}_{t_m}(\cdot) 
     f_{\tau,N}
     \big( \widetilde{Y}^{N,\tau}_{t_m}(\cdot)  +\mathcal{R}^{N,\tau}_{t_m}(\cdot) 
     \big)
     +
     \tau^2
     \big| 
     f_{\tau,N}
     \big(  \widetilde{Y}^{N,\tau}_{t_m}(\cdot) + \mathcal{R}^{N,\tau}_{t_m}(\cdot) \big)
     \big|^2
  \notag \\
&
\leq
     \big| \widetilde{Y}^{N,\tau}_{t_m}(\cdot) \big|^2
     - 
     2 \widetilde{c}_0 \tau
     \big|
     \widetilde{Y}^{N,\tau}_{t_m}(\cdot) + \mathcal{R}^{N,\tau}_{t_m}(\cdot) 
     \big|^2
     +
     2 \widetilde{c}_1 
     \tau
     \Big( 
        1
        + 
       \big|
\mathcal{R}^{N,\tau}_{t_m}(\cdot)
       \big|^{2q} 
     \Big)
  \notag \\
&
\leq
     ( 1 - \widetilde{c}_0 \tau )
     \big| \widetilde{Y}^{N,\tau}_{t_m}(\cdot) \big|^2
     +
     2 \widetilde{c}_0 
     \tau
     \big| \mathcal{R}^{N,\tau}_{t_m}(\cdot) \big|^2
     +
     2 \widetilde{c}_1 
     \tau
      \Big( 
        1 
        +
       \big|
         \mathcal{R}^{N,\tau}_{t_m}(\cdot)
       \big|^{2q} 
      \Big)
  \notag \\
&
\leq
     e^{ - \widetilde{c}_0 \tau }
     \big| \widetilde{Y}^{N,\tau}_{t_m}(\cdot) \big|^2
     +
     \tau
     \cdot
     2 (\widetilde{c}_0 + \widetilde{c}_1 )
     \Big( 
        1 
        +
       \big|
          \mathcal{R}^{N,\tau}_{t_m}(\cdot)
       \big|^{2q} 
      \Big), 
\end{align}
where we employed the assumption \eqref{eq:assume_(u+v)f_tau(u)[weak4AC-25]} in the first inequality, and
the Young inequality that $-4 \widetilde{c}_0 \tau \widetilde{Y}^{N,\tau}_{t_m}(\cdot) \mathcal{R}^{N,\tau}_{t_m}(\cdot) \leq \widetilde{c}_0 \tau (\widetilde{Y}^{N,\tau}_{t_m}(\cdot))^2 + 4 \widetilde{c}_0 \tau ( \mathcal{R}^{N,\tau}_{t_m}(\cdot) )^2$ in the second inequality.
%
%
%
By raising both sides of the above inequlity to the $(4q-2)$-th power and utilizing Lemma \ref{lem:ineq_a+taub[weak4AC-25]} with $\upsilon = \widetilde{c}_0$,
we show
\begin{align}
 &
 \big| 
     \widetilde{Y}^{N,\tau}_{t_m}(\cdot) 
     + 
     \tau f_{\tau,N}
     \big(  \widetilde{Y}^{N,\tau}_{t_m}(\cdot) +\mathcal{R}^{N,\tau}_{t_m}(\cdot) 
     \big)
     \big|^{4q-2}
\notag\\ 
&\leq
   e^{2(q-1) \widetilde{c}_0 \tau }
   \big( 
     e^{-\widetilde{c}_0 \tau} |\widetilde{Y}^{N,\tau}_{t_m}(\cdot)|^2 
    \big)^{2q-1}
   +
   \tau
   \left( 
  \tau^{2q-2} 
  + 
  (1 +  ( \tfrac{2}{ \widetilde{c}_0 } )^{2q-1} )
  e^{ (2q-2)\tau }
 \right)  
   \Big(
     2(\widetilde{c}_0 + \widetilde{c}_1)
     \big(1 + |\mathcal{R}^{N,\tau}_{t_m}(\cdot)|^{2q}\big)
   \Big)^{2q-1}
\notag\\
& \leq    
   e^{ - \widetilde{c}_0 \tau }
   \big| \widetilde{Y}^{N,\tau}_{t_m}(\cdot) \big|^{4q-2} 
     +
     C \tau
     \Big(
     1 + \big|\mathcal{R}^{N,\tau}_{t_m}(\cdot)\big|^{ 2q( 2q-1 ) } 
     \Big).
\end{align}
Integrating the above inequality over $\mathcal{D}$ and taking \eqref{eq::Y^N,tau(L^4q-2_bound)first_step_inprf[weak4AC-25]} into account give  
\begin{equation}
\begin{aligned}
    \big\| \widetilde{Y}^{N,\tau}_{t_{m+1}} \big\|_{L^{4q-2}}^{4q-2}
&
\leq
    e^{ - \widetilde{c}_0 \tau }
    \big\| \widetilde{Y}^{N,\tau}_{t_m} \big\|_{L^{4q-2}}^{4q-2}
  +
     C \tau
    \Big(
      1 + \big\| \mathcal{R}^{N,\tau}_{t_m} \big\|_V^{ (4q-2)q } 
    \Big)  
\\
&
\leq
    e^{ -(m+1) \widetilde{c}_0 \tau }
    \big\| \widetilde{Y}^{N,\tau}_{0} \big\|_{L^{4q-2}}^{4q-2}
  +
     C
    \tau
    \sum_{k=0}^{m}
   e^{ -(m-k) \widetilde{c}_0 \tau }
    \Big(
      1 + \big\| \mathcal{R}^{N,\tau}_{t_m} \big\|_V^{ (4q-2)q } 
    \Big) .
\end{aligned}
\end{equation}
Armed with the bound \eqref{eq:Vbound_mathcal(R)[weak4AC-25]},
by further taking expectations on both sides of the above inequality, 
we similarly show
for any $p \geq 1$,
\begin{equation}\label{eq:bound_Widetilde_Y[weak4AC-25]}
       \sup_{m \in \N_0}
        \big\|
          \widetilde{Y}^{N,\tau}_{t_m} 
        \big\|_{L^p( \Omega; L^{4q-2} )}
    \leq
        e^{- \frac{ \widetilde{c}_0 m \tau}{4q-2}  } 
       \big\| 
         \widetilde{Y}^{N,\tau}_{0}
       \big\|_{L^{p}(\Omega; L^{4q-2} )}
      +
        C
        \sup_{m \in \N_0}
        \| \mathcal{R}^{N,\tau}_{t_m} \|^{q}_{L^{qp}(\Omega;V)}
        < \infty.
    \end{equation}

Secondly, we show the uniform moment bounds of $\widetilde{Y}^{N,\tau}$ in $V$-norm. 
By recalling the Sobolev embedding inequality that $ \| x \|_V \leq C \| A^{\vartheta} x \|$, for $\vartheta \in (\frac{d}{4},1 \wedge \frac{\varrho}{2} ) $, $x \in V$, with $\varrho$ coming from Assumption \ref{assump:X_0(Initial Value)[weak4AC-25]},
one then similarly derives from \eqref{eq:widetilde(Y)^N,tau[weak4AC-25]} that
for all $m \in \N_0$,
\begin{equation}
  \begin{aligned}
   \big\| \widetilde{Y}^{N,\tau}_{t_{m}} \big\|_V
& \leq 
      C
     \Big\| 
         A^\vartheta 
         E( t_{m} )  Y^{N,\tau}_0 
      \Big\|
      +
      C \tau 
            \sum_{k=0}^{m-1}
            \Big\|
               A^\vartheta 
               E( t_{m} - t_{k} )
                F_{\tau,N}
               \big( 
                   \widetilde{Y}^{N,\tau}_{t_k} 
                   +
                   \mathcal{R}^{N,\tau}_{t_k} 
               \big)
            \Big\|
\\
& \leq 
      C 
      \big\| 
        X_0
      \big\|_{ \varrho}
      +
      C \tau 
            \sum_{k=0}^{m-1}
           e^{ - \frac12 \lambda_1 
        ( t_{m} - t_k )}
        ( t_{m} - t_k )^{ -\vartheta } 
         \Big(  
            1
            +
            \big\| 
            \widetilde{Y}^{N,\tau}_{t_k} 
            \big\|^{ 2q-1 }_{ L^{4q-2} }
           + 
            \big\|
               \mathcal{R}^{N,\tau}_{t_k} 
            \big\|_V^{2q-1} 
        \Big),
   \end{aligned}     
\end{equation}
where we
utilized the
assumptions \eqref{eq:assume_|F_tau(u)|_lambda[weak4AC-25]}, \eqref{eq:asuume_|f(u)|[weak4AC-25]} in the last inequality.
By using
Assumption \ref{assump:X_0(Initial Value)[weak4AC-25]} and the bounds 
\eqref{eq:Vbound_mathcal(R)[weak4AC-25]},
\eqref{eq:bound_Widetilde_Y[weak4AC-25]}, 
we get for any $p\geq 1$,
\begin{equation}\label{eq:Vbound_Wiltid(Y)[weak4AC-25]}
\begin{aligned}
        \sup_{m \in \N_0}
        \big\|
          \widetilde{Y}^{N,\tau}_{t_m} 
        \big\|_{L^p( \Omega; V )} 
 &\leq
        C\bigg( 1 
           + 
           \| X_0\|_{L^p( \Omega; \dot{H}^{\varrho} )} 
           +
            \sup_{m \in \N_0}
            \big\|
            \widetilde{Y}^{N,\tau}_{t_m} \big\|^{2q-1}_{L^{(2q-1)p}( \Omega; L^{4q-2} )} 
            +
           \sup_{m \in \N_0}
            \big\|
          \mathcal{R}^{N,\tau}_{t_m} 
        \big\|^{2q-1}_{L^{(2q-1)p}( \Omega; V )} 
        \bigg)
\\
&
        < \infty.
\end{aligned}
\end{equation}

\vspace{1.5mm}
\textbf{Step 4: Uniform moment bounds of  $X^{N,\tau}$.}
\vspace{1.5mm}



For the bound of  
$ \sup_{m \in \N_0}\big\|
          X^{N,\tau}_{t_m} 
        \big\|_{L^p( \Omega; V )}$,
owing to the decomposition \eqref{eq:X_decompo_Y_R[weak4AC-25]}, as well as the bounds \eqref{eq:Vbound_mathcal(R)[weak4AC-25]} and \eqref{eq:Vbound_Wiltid(Y)[weak4AC-25]},
it is evident that for any $p\geq 1$,
   \begin{equation}\label{eq:Vbound_X[weak4AC-25]}
        \sup_{m \in \N_0} 
        \big\|
          X^{N,\tau}_{t_m} 
        \big\|_{L^p( \Omega; V )}
\leq
        \sup_{m \in \N_0} 
        \big\|
          \widetilde{Y}^{N,\tau}_{t_m} 
        \big\|_{L^p( \Omega; V )}
        +
        \sup_{m \in \N_0} 
        \big\|
          \mathcal{R}^{N,\tau}_{t_m} 
        \big\|_{L^p( \Omega; V )}
        <
        \infty.
    \end{equation}

For the bound of
$\sup_{m \in \N_0}
\big\| 
      X^{N,\tau}_{t_m} 
  \big\|_{ L^p(\Omega;\dot{H}^{\kappa}) }
  \kappa \in [0, \gamma)$,
we similarly get for all $
  \kappa \in [0, \gamma), m \in \N_0$,
\begin{align}
    \big\| 
      X^{N,\tau}_{t_{m}} 
    \big\|_{ \kappa }
&   = 
      \bigg\|
         E_N( t_{m} ) 
         X_0
         +
         \tau 
         \sum_{k=0}^{m-1}
         E_N( t_{m} - t_k ) 
         P_N
         F_{\tau,N}
          \big( X^{N,\tau}_{t_k} 
          \big)   
         +
         \mathcal{O}^{N,\tau}_{t_{m}}
      \bigg\|_{ \kappa }
\notag\\
& \leq 
      \| X_0 \|_{ \kappa }
      +
      C
      \tau 
      \sum_{k=0}^{m-1}
        ( t_{m} - t_k )^{ -\frac{ \kappa }{2} }
        e^{ - \frac12 \lambda_1 
            ( t_{m} - t_k )
          }
 \Big( 
         1 
           +
            \big\| 
              X^{N,\tau}_{t_k}  
            \big\|_V^{2q-1} 
        \Big) 
        +
       \big\|
         \mathcal{O}^{N,\tau}_{t_{m}}
      \big\|_{ \kappa },
   \end{align}
where we employed the property \eqref{eq:E(t)_semigroup_property[weak4AC-25]} and
the assumption \eqref{eq:assume_|F_tau(u)|[weak4AC-25]} in the last inequality.
Due to Assumption \ref{assump:X_0(Initial Value)[weak4AC-25]}, Lemma \ref{lem:O^N,tau(bound)[weak4AC-25]} and 
the bound \eqref{eq:Vbound_X[weak4AC-25]},
we finally acquire for any $p\geq 1$, $
  \kappa \in [0, \gamma)$,
\begin{equation}
    \sup_{m \in \N_0}
    \| X^{N,\tau}_{t_m} \|_{L^p(\Omega;\dot{H}^{\kappa}) }
\leq
    C 
    \bigg( 
    1 
    +
    \| X_0 \|_{L^p(\Omega;\dot{H}^{\kappa}) }
    + 
      \sup_{m \in \N_0} 
        \big\|
          X^{N,\tau}_{t_m} 
        \big\|^{2q-1}_{L^{p(2q-1)}( \Omega; V )} 
   +
    \sup_{m \in \N_0} 
     \big\|
         \mathcal{O}^{N,\tau}_{t_{m}}
      \big\|_{L^p(\Omega;\dot{H}^{\kappa}) }
    \bigg)
   <
   \infty.
\end{equation}
The proof is thus completed.
\end{proof}

\section{Uniform-in-time weak convergence analysis and approximations of invariant measures}
\label{sec:convergence[weak4AC-25]}

In this section, we attempt to carry out the uniform-in-time weak convergence analysis of the full-discretization schemes \eqref{eq:full_discretization[weak4AC-25]}.
To achieve this, we require further assumptions as follows.

\begin{assumption}
\label{assume:contractive_or_non-degeneracy[weak4AC-25]}
Assume either 
$
L_f < \lambda_1
$
    or
the covariance operator $Q \in \mathcal{L}(H)$ is invertible, satisfying 
$\|  Q^{-\frac12} (-A)^{-\frac12} \|_{\mathcal{L} } 
< 
\infty. 
$
\end{assumption}
Here the former assumption is called the contractive condition, and the latter one is called the non-degeneracy condition.
Under Assumption \ref{assume:contractive_or_non-degeneracy[weak4AC-25]}, 
an exponential convergence to equilibrium for the SPDE \eqref{eq:considered_SEE[weak4AC-25]} can be attained (see, e.g., \cite[Proposition 3.3]{brehier2022ESAIM}, \cite[Theorem 12.5]{Goldys2006}). 
%


\begin{prop}\label{prop:V-uniform_ergodicity_X(t)[weak4AC-25]}
    Let Assumptions \ref{assump:A(linear_operator)[weak4AC-25]}-\ref{assump:F(Nonlinearity)[weak4AC-25]} 
    and 
    Assumption \ref{assume:contractive_or_non-degeneracy[weak4AC-25]} hold.
    Let $X(t,x), t\geq 0$ be the unique solution of \eqref{eq:considered_SEE[weak4AC-25]} that initiates at $x \in H$.
    Then there exist constants $c > 0$ and $ C > 0$ such that for any $\varphi \in \mathcal{C}_b^1(H)$, $t \geq 0$, and $u, v \in H$, it holds
    \begin{equation}
        \left| 
        \E
        \left[ \varphi( X(t,u) ) \right] 
        - 
        \E\left[ \varphi( X(t,v) ) \right] \right|
        \leq
        C
        \| \varphi \|_{\mathcal{C}_b^1(H)}
        e^{-ct}
        \left( 1 + \|u\|^2 + \|v\|^2 \right).
    \end{equation}  
\end{prop}

Equipped with Proposition \ref{prop:V-uniform_ergodicity_X(t)[weak4AC-25]},
one can show the
existence of the unique invariant measure for SPDE \eqref{eq:considered_SEE[weak4AC-25]},
due to the Doob theorem for $L_f \in \R$ \cite{da2006introduction}.

\begin{thm}
    Let Assumptions \ref{assump:A(linear_operator)[weak4AC-25]}-\ref{assump:X_0(Initial Value)[weak4AC-25]} 
    and 
    Assumption \ref{assume:contractive_or_non-degeneracy[weak4AC-25]} hold.
    Then the SPDE \eqref{eq:considered_SEE[weak4AC-25]} admits a unique invariant measure $\mu$.
\end{thm}



Below, we introduce a continuous version of the full-discretization schemes \eqref{eq:full-dis_sum[weak4AC-25]}, defined by
\begin{equation}\label{eq:continuous_version(full_disc)[weak4AC-25]}
    X^{N,\tau}(t) 
    = 
    E_N(t) X^{N,\tau}_0
    + 
    \int_0^t 
     E_N( t - \lfloor s \rfloor_{\tau} ) P_N F_{\tau,N}( X^{N,\tau}(\lfloor s \rfloor_{\tau}) )  \,\mathrm{d}s
    + 
    \int_{0}^{t} 
    E_N ( t - \lfloor s \rfloor_{\tau} ) P_N \,\mathrm{d}W(s),    
\end{equation}
where $t \geq 0$, and $\lfloor s\rfloor_{\tau} := t_k$ for $s \in [t_k, t_{k+1}), k \in \N_0 $. 
Moreover, we note the process \eqref{eq:continuous_version(full_disc)[weak4AC-25]} satisfies $  X^{N,\tau}(t)=
 X^{N,\tau}_{t_k} $ for $ t = t_k, k \in N_0$ and
\begin{equation}
    d X^{N,\tau}(t)
    =
    - A_N X^{N,\tau}(t) \,\mathrm{d}t
    + E_N( t - \lfloor t \rfloor_{\tau} ) P_N
    F_{\tau,N} ( X_{\lfloor t \rfloor_{\tau} }^{N,\tau} ) \,\mathrm{d}t
    + E_N( t - \lfloor t \rfloor_{\tau}  ) P_N \,\mathrm{d}W(t).
\end{equation}
Thanks to Theorem \ref{thm:X^N,tau(bound)[weak4AC-25]},
the uniform moment bound
$\sup_{t \geq 0} \| X^{N, \tau}(t) \|_{L^p(\Omega;V)}, p \geq 1$ for the continuous extension
is straightforward by standard arguments.
Next
we introduce its H\"older regularity property in negative Sobolev spaces as follows, whose proof is similar to that in \cite[Lemma 4.10]{wang2020efficient}.

\begin{lem}\label{lem:holder-full-discretization[weak4AC-25]}
     Let Assumptions \ref{assump:A(linear_operator)[weak4AC-25]}-\ref{assump:F(Nonlinearity)[weak4AC-25]} and Assumption \ref{assump:F_N,tau[weak4AC-25]} hold. Let $X^{N,\tau}(t), t \geq 0$ be defined by \eqref{eq:continuous_version(full_disc)[weak4AC-25]}. Then for any $p \in [2, \infty)$, 
     $\eta \geq 0$,
     there exists a constant $C(p, \eta, \gamma) > 0$ depending on $p, \eta, \gamma$, such that for any $t > s \geq 0$,
     \begin{equation}
         \left\|
         X^{N,\tau}(t) - X^{N,\tau}(s) \right\|_{L^p(\Omega, \dot{H}^{-\eta} )}
         \leq
         C(p, \eta, \gamma) (t-s)^{ \frac{\gamma+\eta}{2} \wedge \frac12 }.
     \end{equation}
\end{lem}

Further, we introduce another spatial semi-discretization process $ X^{K}(t,x ) ,t \geq 0$ of the spectral Galerkin method \eqref{eq:spatial_disc[weak4AC-25]} that initiates at $x \in H^{K}, K \in \N$.
For any $\varphi \in \mathcal{B}_b(H^{K})$, we define
\begin{equation}
    \nu^{K}(t, x) 
    := 
    \mathbb E\left[ \varphi \big( X^{K} (t,x) \big)\right],
    \quad
    t \geq 0, x \in H^{K}, 
\end{equation}
which is 
the unique solution of the Kolmogorov equation associated to $ X^{K}(t,x ) ,t \geq 0$:
\begin{equation}\label{eq:Kolmogorov_eq_for_U_N(spa)[weak4AC-25]}
\partial_t 
\nu^{K}(t,x) 
 =  
 D \nu^{K}(t,x)
. ( -A x 
   +
  P_{K}
     F(x) ) 
 + 
\frac{1}{2} 
\mathrm{Tr}
\left[ D^2 \nu^{K}(t,x) 
    P_{K} 
    Q^{\frac12} 
    ( P_{K} Q^{\frac12} )^* 
            \right], 
\end{equation}
with $\nu^{K}( 0,\cdot) 
          = 
          \varphi(\cdot)$. 
%
%
We first show that $X^K(t), t \geq 0$ satisfies the uniform moment bound as follows, whose proof follows a slight modification of that in \cite[Lemma 2, Corollary 2]{CUI2021weak} for $q=2$, and is thus omitted.

\begin{lem}
\label{lem:moment-bound-sptailDiscre[weak4AC-25]}
    Let Assumptions \ref{assump:A(linear_operator)[weak4AC-25]}-\ref{assump:X_0(Initial Value)[weak4AC-25]} be fulfilled.
    Let $ X^K(t), t \geq 0$ be the solution of the spectral Galerkin
method
 \eqref{eq:spatial_disc[weak4AC-25]} that initiates at $X_0^{K} \in H^K \cap V, K \in \N$.
    For any $ p\geq 1$, there exists a constant $C( Q,p,q,d ) > 0$ such that 
    \begin{align}
        \sup_{t \geq 0}
        \big\| 
          X^K( t ) 
        \big\|_{ L^p( \Omega ; V ) }
        \leq
        C( Q,p,q,d ) 
        \bigg(1
        +
        \|X_0^K\|^{ \frac{(2q-1)[ (2q-2)((8q-8) - (2q-3)d) + 4]}{ 4-(2q-3)d } }_{ L^{ \frac{(2q-1)[ (2q-2)((8q-8) - (2q-3)d) + 4]p}{ 4-(2q-3)d } } ( \Omega; V)}
        \bigg).
    \end{align}
\end{lem}

In the following lemma, we show the regularity estimates for $\nu^{K}( \cdot, \cdot),K \in \N$, whose proof is based on the Bismut-Elworthy-Li formula and can be found in \cite[Lemma 5, Lemma 6]{CUI2021weak} for $q=2$.

\begin{lem}\label{lem:DU^N(t,X)(regularity)[weak4AC-25]}
Let Assumptions \ref{assump:A(linear_operator)[weak4AC-25]}-\ref{assump:X_0(Initial Value)[weak4AC-25]}
and 
Assumption \ref{assume:contractive_or_non-degeneracy[weak4AC-25]} hold.
For any $\varphi \in \mathcal{C}_b^2(H)$ and $\vartheta_0, \vartheta_1, \vartheta_2 \in [0,1)$, $\vartheta_1 + \vartheta_2 < 1$, there exist constants $c > 0, C( Q, \vartheta_0, \varphi) >0$ and $ C(Q, \vartheta_1, \vartheta_2, \varphi) >0$ such that for $x, y, z \in H^{K}$, $K \in \N$ and $t > 0$,
\begin{align}
\label{eq:DU^N(t,X)(regularity)[weak4AC-25]}
    \Big|D \nu^{K}(t, x). y \Big|
    \leq 
    & C(Q, \vartheta_0, \varphi)
    \Big( 1 
           +
           \sup_{s \in[0, t]} 
           \mathbb{E}\Big[ \big\| 
                            X^{K}(s, x)
                            \big\|_V^{2q}
                      \Big]
    \Big)
    \left( 1 + t^{-\vartheta_0} \right) 
    e^{-c t}
    \| y \|_{-2\vartheta_0},
\\
\label{eq:D^2-U^N(t,X)(regularity)[weak4AC-25]}
    \Big|
    D^2 \nu^{K}(t, x). (y, z)
    \Big| 
    \leq 
    & C(Q, \vartheta_1,\vartheta_2, \varphi) 
    \Big( 1 
           + 
           \sup_{s \in[0, t]}
           \mathbb{E} \Big[\big\|
                            X^{K}(s,x)
                            \big\|_V^{8q-2}
                      \Big]
    \Big)
    \left( 1 + t^{-\vartheta_1 -\vartheta_2} \right)
    e^{-c t}
    \| y \|_{-2\vartheta_1}
    \| z \|_{-2\vartheta_2}.    
\end{align}
\end{lem}

To achieve the desired weak convergence rates of the full-discretization schemes \eqref{eq:full_discretization[weak4AC-25]},
the following commutativity properties of the nonlinearity $F$ in negative Sobolev spaces are required.

\begin{lem}\label{lem:F(negativeSobolev)[weak4AC-25]}
    Let the nonlinear operator $F\colon L^{4q-2}(\mathcal{D}) \rightarrow H, q \geq 1$ satisfy Assumption \ref{assump:F(Nonlinearity)[weak4AC-25]}. Then for any $\vartheta \in (0, 1)$ and $\eta > \max\{ \frac{d}{2}, 1\}$, there exists a constant $C(\vartheta, \eta, q)  >0$ depending on $\vartheta, \eta, q$, such that for any $ u, v \in V \cap \dot{H}^{\vartheta}$, it holds
\begin{equation}
      \left\| F'(\phi) \varsigma \right\|_{-\eta}
      \leq
       C(\vartheta, \eta, q)
       \left( 1 + 
          \max 
          \left\{ \|\phi\|_V,
          \|\phi\|_\vartheta \right\}^{2q-1} 
    \right)
     \| \varsigma \|_{-\vartheta}.
\end{equation}
%

\end{lem}

\begin{proof}
Due to the standard arguments with the Sobolev–Slobodeckij norm (cf. \cite{thomee2007galerkin}) and Assumption \ref{assump:F(Nonlinearity)[weak4AC-25]}, one obtains for any $ \phi \in V \cap \dot{H}^{\vartheta},
\zeta \in V  \cap \dot{H}^{\eta}$,
$\eta > \max\{ \frac{d}{2}, 1\}$,
\begin{equation}
\begin{aligned}
\left\| F'(\phi) \zeta\right\|_\vartheta^2 
  \leq 
  & 
  C \left\|F'(\phi) \zeta\right\|^2
  +
  C 
  \int_\mathcal{D} \int_\mathcal{D} 
  \frac{\left| f'(\phi(x)) \zeta(x) - f'(\phi(y)) \zeta(y) \right|^2 }{ |x-y|^{2 \vartheta + d}} 
  \,\mathrm{d}y \,\mathrm{d}x 
\\
  \leq
  &
  C \left\| F'(\phi) \zeta \right\|^2 
  +
  C \int_\mathcal{D} \int_\mathcal{D} 
  \frac{\left| f'(\phi(x)) (\zeta(x)-\zeta(y)) \right|^2}{ |x-y|^{2 \vartheta + d} } 
  \,\mathrm{d}y \,\mathrm{d}x 
\\
&  + C \int_\mathcal{D} \int_\mathcal{D} 
   \frac{ \left| \left[ f'( \phi(x) )-f'( \phi(y) ) \right] \zeta(y) \right|^2}{ |x-y|^{2 \vartheta+d } } 
   \,\mathrm{d}y \,\mathrm{d}x 
\\
   \leq 
   & 
   C \left\|F'(\phi) \zeta \right\|^2
   +
   C \left\|f'(\phi(\cdot))\right\|_V^2 \cdot\|\zeta\|_{ W^{\vartheta, 2} }^2 
   +
   C \left\| f''(\phi(\cdot))\right\|_V^2 \cdot
  \| \zeta \|_V^2 \cdot \| \phi \|_{W^{\vartheta, 2} }^2 
\\
  \leq
  & C \left( 1 + \| \phi \|_V^{4q-4} \right) \| \zeta \|^2
+ 
C \left( 1 + \| \phi \|_V^{4q-4} \right) \| \zeta \|_\vartheta^2
+
C \left( 1 + \|\phi\|_V^{4q-4} \right) \|\zeta\|_V^2 \cdot \|\phi\|_\vartheta^2 
\\
\leq 
&  
C \left( 1 + \max \left\{ \|\phi\|_V, \|\phi\|_\vartheta \right\}^{4q-2} \right) \left( \|\zeta\|_\vartheta^2 + \|\zeta\|_V^2 \right).
\end{aligned}
\end{equation}
Accordingly, the Sobolev embedding inequality and the fact that $\vartheta \in (0, 1), \eta > \max\{ \frac{d}{2}, 1\}$ imply
\begin{equation}
\begin{aligned}
     \left\| F'(\phi) \varsigma \right\|_{-\eta}
     &
     =
     \sup _{\|\varphi\| \leq 1}
     \Big|
     \Big\langle
       A^{-\frac{\eta}{2}} F'(\phi) \varsigma
       ,
       \varphi 
      \Big\rangle \Big|
\\
     & 
     =
     \sup _{\|\varphi\| \leq 1}
     \Big|
     \Big\langle 
     A^{ -\frac{ \vartheta }{2} } \varsigma 
     ,
     A^{ \frac{ \vartheta }{2} } F'(\phi) A^{-\frac{\eta}{2}} \varphi
     \Big\rangle \Big| 
\\
    & 
    \leq
    \sup_{\|\varphi\| \leq 1 }
    \| \varsigma \|_{-\vartheta} 
    \cdot
    \Big\| 
    F'(\phi) A^{-\frac{\eta}{2}} \varphi \Big\|_\vartheta
\\
    &
    \leq
    \sup _{\|\varphi\| \leq 1}
    \| \varsigma \|_{-\vartheta} 
    \cdot 
    C \left(
       1 
       +
       \max 
       \left\{ \|\phi\|_V , \|\phi\|_\vartheta \right \}^{2q-1} 
       \right)
    \left( \|\varphi\|_{ \vartheta - \eta}
            +
           \big\|A^{ -\frac{\eta}{2} } \varphi\big\|_V 
    \right)
\\
    &
    \leq
    C 
    \sup _{\|\varphi\| \leq 1}
    \| \varsigma \|_{-\vartheta} 
    \cdot 
    \left(
       1 
       +
       \max 
       \left\{ \|\phi\|_V , \|\phi\|_\vartheta \right \}^{2q-1} 
       \right)
    \left\| \varphi \right\|
\\
    &
    \leq
    C
    \left( 1 + 
          \max 
          \left\{ \|\phi\|_V,
          \|\phi\|_\vartheta \right\}^{2q-1} 
    \right)
     \| \varsigma \|_{-\vartheta},
\end{aligned} 
\end{equation}
where the self-adjointness of $A, F'(\varphi)$ was also used.
The proof is thus completed.
%
%
\end{proof}

Equipped with the preceding results, we now state the main convergence result of this paper. 

\begin{thm}[The space-time full error bounds]
\label{thm:uniform_weakerror[weak4AC-25]}
   Let Assumptions \ref{assump:A(linear_operator)[weak4AC-25]}-\ref{assump:X_0(Initial Value)[weak4AC-25]} and Assumptions \ref{assump:F_N,tau[weak4AC-25]}, \ref{assume:contractive_or_non-degeneracy[weak4AC-25]}
   hold.
Let $X(t), t\geq 0$ and $ X_{t_m}^{N,\tau}, m \in \N_0$ be defined by \eqref{eq:spde[weak4AC-25]} and \eqref{eq:full_discretization[weak4AC-25]}, respectively.
   Then for any $\varphi \in \mathcal{C}_b^2(H)$, there exists $C(X_0, Q,q,\varphi, \kappa, \iota) > 0$ such that for $\tau , \lambda_N >0 $, $m, N \in \N $, 
   \begin{equation}
       \Big| 
       \E\Big[
       \varphi ( X( t_m ) ) - \varphi( X_{t_m}^{N,\tau} ) 
       \Big] \Big|
       \leq
       C(X_0, Q,q, \varphi, \kappa, \iota)
       ( 1 + t_m^{-\iota} )
       \left( 
          \tau^{\kappa} + \lambda_N^{-\iota} 
       \right),
   \end{equation}
with $\kappa \in (0, (\gamma \wedge \theta \wedge 1) )$, $\iota \in (0, (\gamma \wedge \rho \wedge 1) )$,
where 
$\gamma$ comes from Assumption \ref{assump:W(noise)[weak4AC-25]} and $\theta, \rho$ are method parameters coming from Assumption
\ref{assump:F_N,tau[weak4AC-25]}.
\end{thm}
\begin{proof}
By introducing the process $ X^{K}(t), t \geq 0$ of spectral Galerkin method \eqref{eq:spatial_disc[weak4AC-25]} that initiates at $ X_0^{K} := P_{K} X_0 \in H^K, N < K \in \N$, we decompose the weak error into the following two terms:
\begin{equation}
\begin{aligned}
&
    \Big| 
      \E\Big[ 
          \varphi ( X( t_m ) ) 
          -
          \varphi \big( X_{t_m}^{N,\tau} \big) 
     \Big] \Big| 
\\
&   \leq
    \Big| 
      \E\Big[
            \varphi ( X( t_m ) ) 
            - 
            \varphi\big( X^{K}( t_m ) \big) 
       \Big] 
    \Big|
    +
    \Big| 
      \E\Big[ 
           \varphi \big( X^{K}( t_m ) \big) 
           -
           \varphi \big( X_{t_m}^{N,\tau} \big) 
       \Big] 
     \Big|.
\end{aligned}
\end{equation}

The first term can be directly estimated by applying the well-established strong convergence with finite time horizon (see, e.g., \cite[Theorem 4.1]{QiWang2019optimal})
and taking the limit $K \rightarrow \infty$, i.e.,
\begin{equation}
    \Big| 
      \E\Big[
            \varphi ( X( t_m ) ) 
            - 
            \varphi\big( X^{K}( t_m ) \big) 
       \Big] 
    \Big|
\leq
C(t_m, \varphi) \lambda_K^{-\frac{\gamma}{2}}
\rightarrow 0,
\quad
K \rightarrow \infty.
\end{equation}
For the second term,
we recall
$\nu^{K}(t, x) = \mathbb E\left[ \varphi \big( X^{K} (t,x) \big)\right], t\geq 0, x\in H^{K}$ and we show
\begin{equation}
\begin{aligned}\label{eq:weak-error-decomp2}
         &
      \Big| 
        \E\Big[
           \varphi
           \big( X^{K}( t_m ) \big) 
           -
           \varphi
           \big( X_{t_m}^{N,\tau} \big) 
      \Big] \Big| \\
      &  =
      \Big| 
        \E\big[ 
             \nu^{K}
             \big( t_m , X_0^{ K } \big)
           \big]
            -
         \E\big[ 
             \nu^{K }
             \big( 0, X_{t_m}^{ N,\tau } \big) 
             \big]
         \Big|    \\
      &  \leq
         \Big| 
          \E\big[  
               \nu^{K } 
               \big( t_m, X_0^{ K } \big) 
            \big]
           -
           \E\big[  
              \nu^{K }
              \big( t_m, X_0^{ N} \big)
            \big] 
         \Big|
         +
         \Big| 
          \E\big[
              \nu^{K } 
              \big( t_m, X_0^{ N,\tau} \big) 
             \big]
           -
            \E\big[
                \nu^{K }
                \big( 0, X_{t_m}^{ N,\tau } \big)
             \big] 
         \Big|,
 \end{aligned}
\end{equation}
where,
for any $\iota \in [0, 1)$,
\begin{equation}
\begin{aligned}
    &
         \left| 
          \E\left[ 
            \nu^{K } 
            \big( t_m, X_0^{ K } \big) 
            \right]
           -
            \E\left[ 
            \nu^{K } 
            \big( t_m, X_0^{N} \big) 
            \right] 
         \right| \\
    &\leq
         \int_0^1
         \left| 
          \E\left[ 
           D \nu^{K } 
           \big( 
             t_m 
             , 
             r X_0^{ K } 
             +
             ( 1 - r ) X_0^{N}
           \big)
           \cdot
           \big( 
           X_0^{ K } - X_0^{N} 
           \big)
             \right]
         \right| 
         \,\mathrm{d} r  \\
    &\leq
         C 
         \Big(
           1 
           +
           \sup_{s \in [0, t_m] } 
           \E\left[ 
                \big\|
                  X^{ K }
                  \big( s, 
                       r X_0^{ K } 
                        +
                        ( 1 - r ) X_0^{N}  
                  \big) 
                 \big\|_V^{2q}
              \right]
          \Big)
          \left( 1 + t_m^{-\iota} \right)
          e^{-c t_m}
          \left\|
             P_{ K }
             ( I - P_N ) 
             X_0 
          \right\|_{-2\iota}  \\
    &\leq
         C ( X_0, Q, q, \varphi, \iota)
          \left( 1 + t_m^{-\iota} \right)
          \lambda_N^{- \iota}.
\end{aligned}
\end{equation}
Here we used Lemma \ref{lem:DU^N(t,X)(regularity)[weak4AC-25]} in the second inequality, and 
the property \eqref{eq:P_N(regularity)[weak4AC-25]} was employed in the last inequality, along with
Lemma \ref{lem:moment-bound-sptailDiscre[weak4AC-25]} and Assumption \ref{assump:X_0(Initial Value)[weak4AC-25]}.
To proceed with the remaining term in \eqref{eq:weak-error-decomp2}, we do the decomposition as follows
\begin{align}
&
         \left| 
          \E\left[
          \nu^{K }
          \big( t_m, X_0^{ N,\tau} \big)
          \right]
           -
           \E\left[
              \nu^{K } 
              \big( 0, X_{t_m}^{ N,\tau } \big)
           \right] 
         \right|
\notag\\
        &=
        \Bigg|
        \sum_{i=0}^{m-1}
        \underbrace{
          \E\left[
            \nu^{K} 
            \left( 
            t_m - t_{i+1}, X^{N,\tau}(t_{i+1}) 
            \right) 
            \right] 
         - 
          \E\left[
          \nu^{K}
          \left( t_m - t_i, X^{N,\tau}(t_{i}) 
          \right)
          \right] 
          }_{ =:I^{(i)} }
        \Bigg| .
\end{align}
Utilizing the It\^o formula and the Kolmogorov equation \eqref{eq:Kolmogorov_eq_for_U_N(spa)[weak4AC-25]} gives for all $i \in \{0,...,m-1\}$,
\begin{equation}
\begin{aligned}
I^{(i)} 
&= 
\mathbb E 
\Bigg[
      \int_{t_i}^{ t_{i+1} } 
        \partial_t \nu^{K}
        \left( t_m - t , X^{N,\tau}(t)  \right)
        \,\mathrm{d}t 
\\
& \qquad
      + 
      \int_{t_i}^{ t_{i+1} }
         D \nu^{K}
         \left( t_m - t , X^{N,\tau}(t) 
         \right) 
         .\left( 
                - A_N X^{N,\tau}(t) 
                +
                E_N (t - t_i) 
                P_N F_{\tau,N}
                \big( X^{{N,\tau}}_{t_i} \big)
          \right) 
       \,\mathrm{d}t  
\\
& \qquad 
      + 
      \frac{1}{2} 
      \int_{t_i}^{ t_{i+1} }
         D^2 \nu^{K}
         \left(
         t_m - t , X^{N,\tau}(t)  
         \right) 
         \sum_{1 \leq j \leq N, j \in \mathbb N} 
         \left(
           E_N(t-t_i) 
           P_N 
           Q^{\frac12} e_j
           , 
           E_N(t-t_i)
           P_N Q^{\frac12} e_j
         \right) 
      \,\mathrm{d}t
\Bigg] 
\\ 
&=
\mathbb E 
\Bigg[ 
 \int_{t_i}^{t_{i+1}}
        - D\nu^{K}
        \left( 
               t_m - t 
               , 
               X^{N,\tau}(t) 
             \right)
             .
             \left(
               -A X^{N,\tau}(t) 
               +
               P_{K} 
               F
               \big( 
               X^{N,\tau}(t)  
               \big)
       \right)
    \,\mathrm{d}t
\\
& \qquad
      -\frac{1}{2}
     \int_{t_i}^{t_{i+1}}
          \sum_{1 \leq j \leq K, j \in \mathbb{N} } 
         D^2 \nu^{K}
         \left(
           t_m - t 
           , 
           X^{N,\tau}(t)  
         \right)
          .
          \left(
            P_{K} 
            Q^{\frac12} e_j, 
            P_{K} 
            Q^{\frac12} e_j
          \right) 
     \,\mathrm{d}t 
\\
& \qquad
      + 
      \int_{t_i}^{ t_{i+1} }
         D \nu^{K}
         \left(
            t_m - t , X^{N,\tau}(t)  
         \right) 
         .
         \left( 
                - A_N X^{N,\tau}(t) 
                +
                E_N(t - t_i) 
                P_N F_{\tau,N}
                \big(
                X^{{N,\tau}}_{t_i} 
                \big)
          \right) 
       \,\mathrm{d}t  
\\ 
& \qquad  
      + 
      \frac{1}{2} 
      \int_{t_i}^{ t_{i+1} }
         D^2 \nu^{K}
         \left(
           t_m - t , X^{N,\tau}(t) 
         \right) 
         \sum_{1 \leq j \leq N, j \in \mathbb N} 
         \left(
            E_N( t - t_i ) 
            P_N Q^{\frac12} e_j, E_N(t-t_i) P_N Q^{\frac12} e_j
         \right) 
      \,\mathrm{d}t
\Bigg] 
\\  
&= \mathbb E
      \Bigg[ 
             \int_{t_i}^{t_{i+1}} 
             D \nu^{K} 
             \left( 
             t_m - t , X^{N,\tau}(t) 
             \right)
             . \left( 
                       E_N(t - t_i) 
                       P_N 
                       F_{\tau,N}
                    (X^{{N,\tau}}_{t_i} )
                       -
                       P_{K}
                       F \left( X^{N,\tau}(t) \right) 
                \right) 
              \,\mathrm{d}t 
        \Bigg]
\\  
&\quad
    + 
      \mathbb E
      \Bigg[  \frac12
              \int_{t_i}^{t_{i+1}}
              \sum_{j \in \mathbb N} 
              D^2 \nu^{K}
              \left( 
              t_m - t , X^{N,\tau}(t)  
              \right)
              .
              \Big( 
                \big(
                E_N( t-t_i ) P_N Q^{\frac12} e_j,
                E_N( t-t_i ) P_N Q^{\frac12} e_j
                \big)
\\&\qquad \quad  
                -
                \big(
                P_{K} Q^{\frac12} e_j,
                P_{K} Q^{\frac12} e_j
                \big)
              \Big) 
            \,\mathrm{d}t
       \Bigg]
\\
& =:  I_1^{ (i) } + I_2^{ (i) }.
\end{aligned}
\end{equation}
For the term $I_{1}^{(i)}$, we do the decomposition
$ I_1^{ (i) } = I_{1,1}^{ (i) } + I_{1,2}^{ (i) } + I_{1,3}^{ (i) } + I_{1,4}^{ (i) } $, where
\begin{align}
    I_{1,1}^{ (i) }
   & :=
    \mathbb E
      \Bigg[ 
             \int_{t_i}^{t_{i+1}} 
             D \nu^{K} 
             \left(
               t_m - t , X^{N,\tau}(t) 
              \right)
             . \left( 
                       E_N(t - t_i) - I )
                       P_N
                       F_{\tau,N}
                       \big( X^{{N,\tau}}_{t_i} 
                       \big)
                \right) 
              \,\mathrm{d}t 
        \Bigg] ,
\\
    I_{1,2}^{ (i) }
   & :=
     \mathbb E
      \Bigg[ 
             \int_{t_i}^{t_{i+1}} 
             D \nu^{K} 
             \left( t_m - t , X^{N,\tau}(t) \right)
             . \left( 
                       P_N
                       F_{\tau,N}
                       \big( X^{{N,\tau}}_{t_i} 
                       \big)
                       -
                       P_N F 
                       \big( X^{{N,\tau}}_{t_i} 
                       \big)
                \right) 
              \,\mathrm{d}t 
        \Bigg]  ,
\\
    I_{1,3}^{ (i) }
   & :=
   \mathbb E
      \Bigg[ 
             \int_{t_i}^{t_{i+1}} 
             D \nu^{K} 
             \left( 
             t_m - t , X^{N,\tau}(t) 
             \right)
             . \left( 
                       P_N F 
                       \big( X^{{N,\tau}}_{t_i} 
                       \big)
                       -
                       P_{K} F 
                       \big( X^{{N,\tau}}_{t_i} 
                       \big)
                \right) 
             \,\mathrm{d}t 
        \Bigg]  ,
\\
   I_{1,4}^{ (i) }
   & :=
    \mathbb E
      \Bigg[ 
             \int_{t_i}^{t_{i+1}} 
             D \nu^{K} 
             \left( t_m - t , X^{N,\tau}(t) 
             \right)
             . \left( 
                      P_{K} F  \big( X^{{N,\tau}}_{t_i} 
                      \big)
                       -
                       P_{K} F  \big( X^{N,\tau}(t) \big) 
                \right) 
            \,\mathrm{d}t 
        \Bigg].
\end{align}
We start by estimating the term $I_{1,1}^{ (i) }$. Using Lemmas \ref{lem:moment-bound-sptailDiscre[weak4AC-25]}, \ref{lem:DU^N(t,X)(regularity)[weak4AC-25]}, Theorem \ref{thm:X^N,tau(bound)[weak4AC-25]} and the property \eqref{eq:E(t)_semigroup_property[weak4AC-25]} yields that for any $\kappa \in [0, 1)$,
\begin{equation}\label{eq:I^i_11[weak4AC-25]}
\begin{aligned}
      \big| I_{1,1}^{ (i) } \big|
      & \leq
        C( Q, \kappa, \varphi )
        \int_{ t_i }^{ t_{i+1} }
        \left(
              1 
              +
              (t_m - t)^{-\kappa} 
        \right)
        e^{-c(t_m - t)}
        \bigg(
            1 + 
              \sup_{ s \in [ 0, t_m - t ] } 
              \E\left[ 
                \left\| 
                X^{K} 
                \big( s, X^{N, \tau}(t) \big) 
                \right\|_V^{2q} 
                \right]
        \bigg)
\\ &\qquad \qquad \qquad \qquad \qquad \qquad
        \cdot
        \E\left[
             \left\| 
               A^{-\kappa} 
               ( E( t - t_i ) - I ) 
               P_N
               F_{\tau,N}
               \big( X_{t_i}^{N,\tau} \big)
            \right\|
        \right]
        \,\mathrm{d}t
\\
      & \leq
        C( X_0 ,Q, q, \kappa, \varphi)
        \tau^{ \kappa }
        \int_{ t_i }^{ t_{i+1} }
        \left(
           1
           +
           (t_m - t)^{-\kappa} 
        \right)
        e^{-c (t_m - t)}
        \,\mathrm{d}t
    \cdot
        \E\left[
            \big\|
            F_{\tau,N}
            \big( X_{t_i}^{N,\tau} \big) 
            \big\|
        \right]
\\ 
      & \leq
        C( X_0 , Q, q, \kappa, \varphi )
        \tau^{ \kappa }
        \int_{ t_i }^{ t_{i+1} }
        \left(
           1 
           +
           (t_m - t)^{-\kappa} 
        \right)
        e^{ -c(t_m - t) }
            \,\mathrm{d}t
\cdot
        \left(
        1 +
        \E\left[
            \big\|
              X_{t_i}^{N,\tau} 
            \big\|_V^{ 2q-1 } 
          \right]
        \right) ,
\end{aligned} 
\end{equation}
where the assumptions \eqref{eq:assume_|F_tau(u)|_lambda[weak4AC-25]} and \eqref{eq:asuume_|f(u)|[weak4AC-25]} were used in the last inequality.
In the same manner, we get
\begin{equation}\label{eq:I^i_12[weak4AC-25]}
\begin{aligned}
    \big|  I_{1,2}^{ (i) }\big|
      & \leq
        C(X_0, Q,q, \varphi)
        \int_{ t_i }^{ t_{i+1} }
        e^{-c(t_m - t)}
        \,\mathrm{d}t
        \cdot
        \E\left[ 
            \left\| 
              P_N
              \left( 
              F_{\tau,N}
              \big( X^{{N,\tau}}_{t_i} \big)
                 -
               F
               \big( X^{{N,\tau}}_{t_i}
                 \big)
               \right) 
               \right\| \right]
\\
      & \leq
        C(X_0, Q,q,\varphi)
        (\tau^{\theta}+\lambda_N^{-\rho})
        \int_{ t_i }^{ t_{i+1} }
        e^{-c(t_m - t)}
        \,\mathrm{d}t
        \cdot
            \E\left[ 
            \Big(
               1 +
               \big\| X_{t_i}^{N,\tau} 
               \big\|_V^{ \widetilde{l} } 
             \Big)
          \big\|F \big(X_{t_i}^{N,\tau} \big) \big\|
             \right]
\\ 
      & \leq
        C(X_0, Q,q, \varphi)
        (\tau^{\theta}+\lambda_N^{-\rho})
        \int_{ t_i }^{ t_{i+1} }
        e^{-c(t_m - t)}
        \,\mathrm{d}t
        \cdot
        \left( 1 +
            \E\Big[ 
                 \big\| X_{t_i}^{N,\tau} 
                \big\|_V^{ \widetilde{l} +2q-1}
             \Big]
        \right)
,
\end{aligned}  
\end{equation}
where the assumption \eqref{eq:assume_|F_tau(u)-F(u)|[weak4AC-25]} was used in the second inequality.
With regard to the term $I_{1,3}^{ (i) }$,
by taking
Lemmas \ref{lem:moment-bound-sptailDiscre[weak4AC-25]}, \ref{lem:DU^N(t,X)(regularity)[weak4AC-25]}, Theorem \ref{thm:X^N,tau(bound)[weak4AC-25]} and the assumption \eqref{eq:asuume_|f(u)|[weak4AC-25]}
into account,
and further using the property \eqref{eq:P_N(regularity)[weak4AC-25]}, we
arrive at for all $\iota \in [0,1)$,
\begin{equation}\label{eq:I^i_13[weak4AC-25]}
\begin{aligned}
      \big|  I_{1,3}^{ (i) } \big|
        &
        \leq
        C(  X_0 ,Q, q, \iota, \varphi)
        \int_{ t_i }^{ t_{i+1} }
        \left( 
          1
          + 
          (t_m - t)^{-\iota}
        \right)
        e^{-c(t_m - t)}
        \E\left[
            \left\| 
               A^{-\iota} 
               P_{K} ( P_N - I ) 
               F
               \big( X^{{N,\tau}}_{t_i}
               \big) 
            \right\| \right]
        \,\mathrm{d}t
\\
        &
        \leq
        C(  X_0 ,Q, q, \iota, \varphi)
         \lambda_N^{-\iota}
        \int_{ t_i }^{ t_{i+1} }
        \left( 
           1 
           + 
           (t_m - t)^{-\iota} 
        \right)
        e^{-c(t_m - t)}
        \E\left[
          \big\| 
          F
          \big( X^{{N,\tau}}_{t_i} \big) 
          \big\| 
          \right]
        \,\mathrm{d}t
\\ 
        &
        \leq
        C(  X_0 ,Q, q, \iota, \varphi)
         \lambda_N^{-\iota}
        \int_{ t_i }^{ t_{i+1} }
        \left(  
           1 
           +
           ( t_m - t )^{-\iota}
        \right)
        e^{ -c( t_m - t ) }
        \left( 
          1 
          + 
          \E\Big[
          \big\| X^{{N,\tau}}_{t_i} \big\|_{ V}^{2q-1}\Big]
        \right)
        \,\mathrm{d}t.
\end{aligned}  
\end{equation}
We are now in a position to estimate
the term $I_{1,4}^{ (i) }$, which must be handled more carefully.
In view of \eqref{eq:D^2-U^N(t,X)(regularity)[weak4AC-25]},
Assumption \ref{assump:F(Nonlinearity)[weak4AC-25]},
Lemma \ref{lem:moment-bound-sptailDiscre[weak4AC-25]} and Theorem \ref{thm:X^N,tau(bound)[weak4AC-25]},
one obtains
\begin{equation}\label{eq:I^i_14[weak4AC-25]}
\begin{aligned}
&
\big| I_{1,4}^{ (i) } \big|
\leq
\Bigg|
    \mathbb E
      \Bigg[ \!
             \int_{t_i}^{t_{i+1}} 
             \!\!\!
             \Big(
             \!
             D \nu^{K} 
             \big(
             t_m - t
             ,
             X^{N,\tau}(t) 
             \big)
             -
             D \nu^{K} 
             \big(
             t_m - t 
             ,
             X^{N,\tau}_{t_i} 
             \big)
             \Big)
             \Big( 
                      P_{K} 
                      F \big( X^{{N,\tau}}_{t_i} 
                      \big)
                      -
                       P_{K}
                       F  \big( X^{N,\tau}(t) \big)
                \Big) 
            \mathrm{d}t 
        \Bigg] 
\Bigg|
\\
&\qquad
+
\Bigg|
   \mathbb E
      \Bigg[ 
             \int_{t_i}^{t_{i+1}} 
             D \nu^{K} 
             \left( t_m - t , X^{N,\tau}_{t_i} 
             \right)
             . \left( 
                      P_{K} F  \big( X^{{N,\tau}}_{t_i} 
                      \big)
                       -
                       P_{K} F  \big( X^{N,\tau}(t) \big) 
                \right) 
            \,\mathrm{d}t 
        \Bigg] 
\Bigg|
\\
&\leq
C (X_0, Q, q, \varphi)
\!\!
\int_{t_i}^{t_{i+1}}
\!\!\!\!
e^{-c(t_m-t)}
 \mathbb E
 \Big[  
 \big\|
 X^{N,\tau}(t)
 -
 X^{N,\tau}_{t_i} 
 \big\|^2
 \Big(
 1 \!
 + \!
 \big\|
X^{N,\tau}(t) 
\big\|_V^{2q-2}
\!
 +
 \big\|
X^{N,\tau}(t_i) 
\big\|_V^{2q-2}
\Big)
 \Big]
 \mathrm{d}t
 +
 J^{(i)}
\\
&\leq
C (X_0, Q, q, \varphi)
\tau^{\gamma \wedge 1}
\int_{t_i}^{t_{i+1}}
e^{-c(t_m-t)}
 \,
 \mathrm{d}t
 +
 J^{(i)},
\end{aligned}
\end{equation}
where we utilized the H\"older inequality and Lemma \ref{lem:holder-full-discretization[weak4AC-25]} in the last inequality and denote
\begin{equation}
    J^{(i)}
    :=
    \Bigg|
   \mathbb E
      \Bigg[ 
             \int_{t_i}^{t_{i+1}} 
             D \nu^{K} 
             \left( t_m - t , X^{N,\tau}_{t_i} 
             \right)
             . \left( 
                      P_{K} F  \big( X^{{N,\tau}}_{t_i} 
                      \big)
                       -
                       P_{K} F  \big( X^{N,\tau}(t) \big) 
                \right) 
            \,\mathrm{d}t 
        \Bigg] 
\Bigg|.
\end{equation}
Using the Taylor expansion we further split it as follows:
 \begin{equation}
 \begin{aligned}
  &   J^{(i)} 
     \leq
    \Bigg|
   \mathbb E
      \Bigg[ 
             \int_{t_i}^{t_{i+1}} 
             D \nu^{K} 
             \left( t_m - t , X^{N,\tau}_{t_i} 
             \right)
             . \left( 
                      P_{K} F'
                      \big(
                      X^{N,\tau}_{t_i} 
                      \big)
                      \left( 
                      X^{N,\tau}(t)
                      -
                      X^{N,\tau}_{t_i}
                      \right) 
                \right) 
            \,\mathrm{d}t 
        \Bigg] 
\Bigg|
\\
&
+
    \Bigg|
   \mathbb E
      \Bigg[ 
             \int_{t_i}^{t_{i+1}} 
             D \nu^{K} 
             \left( t_m - t , X^{N,\tau}_{t_i} 
             \right)
             . \left( 
             \int_0^1
                      P_{K} F''
                      ( \chi )
                      \Big(
                      X^{N,\tau}(t)
                      -
                      X^{N,\tau}_{t_i}
                      ,
                      X^{N,\tau}(t)
                      -
                      X^{N,\tau}_{t_i}
                      \Big)
                      (1 - r)
                      \mathrm{d}
                      r 
                \right) 
            \mathrm{d}t 
        \Bigg] 
\Bigg|
\\
&
=:
J^{(i)}_1
+
J^{(i)}_2,
 \end{aligned}   
 \end{equation}
where we denote 
$    \chi:= 
X^{N,\tau}_{t_i} 
                      +
                      r
                      (
                      X^{N,\tau}(t)
                      -
                      X^{N,\tau}_{t_i}
                      )
$ for simplicity.
For the term $J^{(i)}_1$, 
by noticing that 
$
 X^{N,\tau}(t)
=
E_N( t - t_i )
X^{N,\tau}_{t_i}
+
( t - t_i )
E_N( t - t_i )
P_N
F_{\tau,N}( X^{N,\tau}_{t_i} )
+
E_N( t - t_i )
(
W(t)
-
W(t_i)
),
$
and employing Lemmas \ref{lem:moment-bound-sptailDiscre[weak4AC-25]}-\ref{lem:F(negativeSobolev)[weak4AC-25]}
and Theorem \ref{thm:X^N,tau(bound)[weak4AC-25]}, one acquires
\begin{equation}\label{eq:J_1[weak4AC-25]}
\begin{aligned}
&
    J^{(i)}_1
\leq
    \Bigg|
   \mathbb E
      \Bigg[ 
             \int_{t_i}^{t_{i+1}} 
             D \nu^{K} 
             \left( t_m - t , X^{N,\tau}_{t_i} 
             \right)
             . \left( 
                      P_{K} F'
                      \big(
                      X^{N,\tau}_{t_i} 
                      \big)
                      \big(
                      (
                      E_N( t- t_i )
                      -
                      I
                      )
                      X^{N,\tau}_{t_i}
                      \big) 
                \right) 
            \,\mathrm{d}t 
        \Bigg] 
      \Bigg|
\\&\qquad
        +
      \Bigg|
       \mathbb E
      \Bigg[ 
             \int_{t_i}^{t_{i+1}} 
             D \nu^{K} 
             \left( t_m - t , X^{N,\tau}_{t_i} 
             \right)
             . \left( 
             ( t- t_i )
                      P_{K} F'
                      \big(
                      X^{N,\tau}_{t_i} 
                      \big)
                      E_N( t- t_i )
                      P_N
                     F_{\tau,N}
                     (X^{N,\tau}_{t_i})
                \right) 
            \,\mathrm{d}t 
        \Bigg]     
\Bigg|
\\
    &
    \leq
    C(X_0, Q, q, \eta, \varphi)
    \int_{t_i}^{t_{i+1}}
    e^{-c (t_m - t) }
    \big(
    1 + (t_m - t)^{-\frac{\eta}{2}} 
    \big)
       \mathbb E
       \Big[
        \big\|
               F'(X^{N,\tau}_{t_i})
               (
                      E_N( t- t_i )
                      -
                      I
                      )
                      X^{N,\tau}_{t_i}
              \big\|_{-\eta}
        \Big]
         \,\mathrm{d}t 
\\&\qquad
 +
      C(X_0, Q, q, \varphi)
      \tau
    \int_{t_i}^{t_{i+1}}
    e^{-c (t_m - t) }   
       \mathbb E
       \Big[
    \big\|F'(X^{N,\tau}_{t_i})\big\|_V
    \big\|F(X^{N,\tau}_{t_i})\big\|
    \Big]
    \,\mathrm{d}t 
\\
    &
    \leq
    C(X_0, Q, q, \eta, \varphi)
    \int_{t_i}^{t_{i+1}}
    e^{-c (t_m - t) }
    \big(
    1 + (t_m - t)^{-\frac{\eta }{2} } 
    \big)
       \mathbb E
\Big[
       \Big( 
       1
       +
       \big\| X^{N,\tau}_{t_i} \big\|_V^{2q-1}
       +
       \big\| X^{N,\tau}_{t_i} \big\|_\kappa^{2q-1}
       \Big)
      \big\| X^{N,\tau}_{t_i} \big\|_{\kappa}  
\\&\
\cdot
        \|
             A^{ -\kappa }
               ( E_N( t- t_i )
                 -
                 I
                )
         \|_{\mathcal{L}(H)}  
\Big]
         \,\mathrm{d}t 
 +
      C(X_0, Q, q, \varphi)
      \tau
    \int_{t_i}^{t_{i+1}}
    e^{-c (t_m - t) }   
       \Big(
       1 + 
       \mathbb E
       \Big[
    \|X^{N,\tau}_{t_i}\|_V^{4q-3}
    \Big]
\Big)
    \,\mathrm{d}t 
\\
&
\leq
      C(X_0, Q, q,\eta, \varphi)
      \tau^{\kappa}
    \int_{t_i}^{t_{i+1}}
    e^{-c (t_m - t) }    
      \big(
      1 + (t_m - t)^{-\frac{\eta}{2} } 
      \big)
    \,\mathrm{d}t ,
\end{aligned}
\end{equation}
where we set $  (\tfrac{d}{2} \vee 1)  < \eta <2 $,
$\kappa < \gamma \wedge 1$,
and also used the property \eqref{eq:E(t)_semigroup_property[weak4AC-25]}.
Moreover,
by using 
 Lemmas \ref{lem:holder-full-discretization[weak4AC-25]}-\ref{lem:DU^N(t,X)(regularity)[weak4AC-25]},
   Theorem \ref{thm:X^N,tau(bound)[weak4AC-25]}, Assumption \ref{assump:F(Nonlinearity)[weak4AC-25]} and the H\"older inequality,
for the term $J^{(i)}_2$ we deduce
\begin{equation}\label{eq:J_2[weak4AC-25]}
 \begin{aligned}
  J^{(i)}_2 
     &
     \leq
C(X_0, Q, q, \vartheta, \varphi)
 \int_{t_i}^{t_{i+1}} 
\big( 1 + (t_m - t)^{-\vartheta} \big)
e^{-c (t_m - t) }
\\&\qquad
\cdot
\left\| 
             \int_0^1
                      P_{K} F''
                      ( \chi )
                      \Big(
                      X^{N,\tau}(t)
                      -
                      X^{N,\tau}_{t_i}
                      ,
                      X^{N,\tau}(t)
                      -
                      X^{N,\tau}_{t_i}
                      \Big)
                      (1 - r)
                      \mathrm{d}
                      r 
                \right\|_{-2\vartheta}
          \,\mathrm{d}t
\\
     &
     \leq
C(X_0, Q, q, \vartheta, \varphi)
 \int_{t_i}^{t_{i+1}} 
\big( 1 + (t_m - t)^{-\vartheta} \big)
e^{-c (t_m - t) }
\\&\qquad
\cdot
   \mathbb E
      \Big[ 
            \big\|
             X^{N,\tau}(t)
             -
             X^{N,\tau}_{t_i}
            \big\|^2
        \big(
        1 
        + 
        \| X^{N,\tau}(t) \|_V^{2q-2}
        +
        \| X^{N,\tau}_{t_i} \|_V^{2q-2}
        \big)
        \Big] 
    \,\mathrm{d}t
\\
&
     \leq
C(X_0, Q, q, \vartheta, \varphi)
\tau^{\gamma \wedge 1}
 \int_{t_i}^{t_{i+1}} 
\big(
1 + (t_m - t)^{-\vartheta} 
\big)
e^{-c (t_m - t) }
    \,\mathrm{d}t,
 \end{aligned}   
 \end{equation}
for any
$\vartheta \in (\frac{d}{4}, 1)$.
Inserting \eqref{eq:J_1[weak4AC-25]}-\eqref{eq:J_2[weak4AC-25]} into \eqref{eq:I^i_14[weak4AC-25]}, 
gathering \eqref{eq:I^i_11[weak4AC-25]}-\eqref{eq:I^i_14[weak4AC-25]} together and utilizing Theorem \ref{thm:X^N,tau(bound)[weak4AC-25]}, one
concludes that for any $ \kappa < ( \gamma \wedge \theta \wedge 1 )$, $\iota < \rho \wedge 1$,
there exists $\delta \in (0,1)$ such that

\begin{equation}
\big|  I_1^{ (i) } \big|
    \leq 
    C(  X_0, Q, q, \varphi, \kappa, \iota ) 
    ( \tau^{\kappa }
      +
      \lambda_N^{ -\iota  }  )
    \int_{ t_i }^{ t_{i+1} }
        \left(
        1 + (t_m - t)^{-1 + \delta}
        \right)
        e^{-c(t_m - t)}
    \,\mathrm{d}t.
\end{equation}

For the term $I_2^{ (i) }, i \in \{0,...,m-1\}$,
we proceed in the same way as above.
Using Lemmas \ref{lem:moment-bound-sptailDiscre[weak4AC-25]}, \ref{lem:DU^N(t,X)(regularity)[weak4AC-25]}, Assumption \ref{assump:X_0(Initial Value)[weak4AC-25]} and Theorem \ref{thm:X^N,tau(bound)[weak4AC-25]},
we get for any $\varsigma \in \big[ (\gamma-1) \vee 0, \gamma \big)$, i.e., $0\leq 1-\gamma+\varsigma <1$, $\varsigma \geq 0$,
\begin{equation}
\begin{aligned}
&
\big| I_2^{ (i) } \big|
=  \Bigg| 
        \int_{t_i}^{t_{i+1}}
            \sum_{j \in \N} 
                \E \bigg[
                D^2 \nu^{K}
                \left( 
                t_m - t, X^{N,\tau}(t)
                \right) 
        \bigg(
           \Big( 
                E_N( t - t_i )
                P_N
                Q^{\frac12} e_j
                ,
                (  E_N( t - t_i )
                   P_N
                   -
                   P_{K}
                ) Q^{\frac12} e_j
           \Big)
\\
        & \qquad \qquad \qquad \qquad
        +
        \Big(  E_N( t - t_i )
                P_N
                Q^{\frac12} e_j
                ,
                P_{K}
                Q^{\frac12} e_j
           \Big) 
         -
        \Big(   P_{K}
                Q^{\frac12} e_j
                ,
                P_{K}
                Q^{\frac12} e_j
           \Big) 
        \bigg)
        \bigg] 
     \,\mathrm{d}t
   \Bigg| 
\\
& 
 \leq
       \Bigg| 
        \int_{t_i}^{t_{i+1}}
            \sum_{j \in \N} 
                \E \bigg[
                D^2 \nu^{K}
                \left( 
                t_m - t, X^{N,\tau}(t)
                \right) 
                \Big( 
                 E_N( t - t_i )
                P_N
                Q^{\frac12} e_j
                ,
                 (  E_N( t - t_i )
                   P_N
                   -
                   I
                )  
                 P_{K}
                Q^{\frac12} e_j
                \Big)
                          \bigg] 
            \,\mathrm{d}t
        \Bigg| 
\\ 
& \quad
    +
        \bigg|
        \int_{t_i}^{t_{i+1}} 
        \sum_{j \in \N} 
        \E 
        \Big[ D^2 \nu^{K}
        \big( t_m - t, X^{N,\tau}(t) \big)      
        \Big( 
           ( E_N( t - t_i ) P_N - I ) 
            P_{K}
           Q^{\frac12} e_j
           , 
           P_{K} 
           Q^{\frac12} e_j
        \Big) \Big] 
       \,\mathrm{d}t
        \bigg| 
\\
 & 
 \leq 
        C (Q, \varphi, \varsigma ) 
        \int_{t_i}^{t_{i+1}} 
\sum_{j \in \N} 
        e^{-c( t_m - t)}
  \bigg(
            1 + 
              \sup_{ s \in [ 0, t_m - t ] } 
              \E\left[ 
                \left\| 
                X^{K} 
                \big( s, X^{N, \tau}(t) \big) 
                \right\|_V^{8q-2} 
                \right]
        \bigg)
\notag
\\  &\qquad \qquad \qquad 
   \cdot
      \big( 
          1 + ( t_m - t)^{ - (1 - \gamma +\varsigma )}
        \big)
{\color{black}
        \big\|
          A^{- (\frac{1 -\gamma}{2} +\varsigma) }
          ( I - E_N( t - t_i ) ) Q^{\frac12} e_j
        \big\|
        \| A^{- \frac{1 -\gamma}{2} } Q^{\frac12} e_j \|
}
       \,\mathrm{d}t
\\
 & 
\leq
        C (X_0, Q,q, \varphi,\varsigma ) 
        \int_{t_i}^{t_{i+1}} 
        e^{-c( t_m - t)}
        \big( 
          1 + ( t_m - t)^{-(1 - \gamma +\varsigma )}
        \big)
        \big\| 
         A^{\frac{\gamma-1}{2}} Q^{\frac12}
       \big\|_{\mathcal{L}_2}^2 
        \big\|
          A^{- \varsigma }
          ( I - E_N( t - t_i ) ) 
        \big\|_{\mathcal{L}(H)} 
       \,\mathrm{d}t
\\
    & \leq 
        C (X_0, Q,q, \varphi,\varsigma ) 
           \left(
           \lambda_N^{-\varsigma }
           +
           \tau^{\varsigma} 
            \right) 
            \int_{t_i}^{t_{i+1}}
            e^{-c( t_m - t)}
            \left( 1 + ( t_m - t )^{ -1 + \gamma -\varsigma} 
            \right) 
           \,\mathrm{d}t,
\end{aligned}   
\end{equation}
where we used the fact
$ \|  A^{- \varsigma }
          ( I - E_N( t - t_i ) ) 
   \|_{\mathcal{L}(H)} 
\leq
\|  A^{- \varsigma } ( I - P_N ) 
 \|_{\mathcal{L}(H)} 
+
\|  A^{- \varsigma} ( I - E( t - t_i ) ) 
\|_{\mathcal{L}(H)} 
$ 
and the properties
\eqref{eq:E(t)_semigroup_property[weak4AC-25]}, \eqref{eq:P_N(regularity)[weak4AC-25]}
in the last inequality,
as well as Assumption \ref{assump:W(noise)[weak4AC-25]}. 
After summing over $i$ for both $I_1^{(i)}$ and $I_2^{(i)}$,
the desired result is thus obtained.
\end{proof}

By taking $\theta = \rho =1$ and $\alpha < 1- \tfrac{d}{4} $, the obtained weak convergence rates coincide with those in \cite{CUI2021weak} obtained for the backward Euler scheme. For example, in the space-time white noise case, the weak convergence rate is nearly order $O(\tau^{\frac12} + \lambda_N^{-\frac12})$ and in the trace-class noise case  the weak convergence rate is nearly order $O(\tau + \lambda_N^{-1})$.

In light of Theorem \ref{thm:uniform_weakerror[weak4AC-25]}, along with the exponential ergodicity of SPDE \eqref{eq:considered_SEE[weak4AC-25]} as established in Proposition \ref{prop:V-uniform_ergodicity_X(t)[weak4AC-25]}, the following corollary is immediately derived.

\begin{cor}\label{cor:convergence_full-invariant[weak4AC-25]}
   Let Assumptions \ref{assump:A(linear_operator)[weak4AC-25]}-\ref{assump:X_0(Initial Value)[weak4AC-25]} and Assumptions \ref{assump:F_N,tau[weak4AC-25]}, \ref{assume:contractive_or_non-degeneracy[weak4AC-25]}
   hold.
For $\mu$ being the unique invariant measure of SPDE \eqref{eq:considered_SEE[weak4AC-25]} and $ X_{t_m}^{N,\tau}, m \in \N_0$ defined by \eqref{eq:full_discretization[weak4AC-25]},
there exist constants $c>0$, $C(X_0, Q,q, \varphi, \kappa, \iota) > 0$ such that 
for any $\varphi \in \mathcal{C}_b^2(H)$,$\tau >0,N \in \N $ and large $M \in \N$, it holds
   \begin{equation}
       \left| 
       \E\left[
           \varphi( X_{t_M}^{N,\tau} )
       \right]
       -
       \int_{H} 
         \varphi 
       \,\mathrm{d}\mu 
        \right|
       \leq
       C(X_0, Q,q, \varphi, \kappa, \iota)
       \left( 
          \tau^{\kappa} 
          + 
          \lambda_N^{-\iota} 
          +
          e^{-c M \tau}
       \right),
   \end{equation}
with $\kappa \in (0, (\gamma \wedge \theta \wedge 1) )$, $\iota \in (0, (\gamma \wedge \rho \wedge 1) )$,
where 
$\gamma$ comes from Assumption \ref{assump:W(noise)[weak4AC-25]} and $\theta, \rho$ 
are method parameters coming from Assumption \ref{assump:F_N,tau[weak4AC-25]}.
\end{cor}

In what follows, we prove that
the full-discretization scheme \eqref{eq:full_discretization[weak4AC-25]} possesses a unique invariant measure $\mu^{N, \tau}$ and thus the convergence rate between $\mu$ and $\mu^{N, \tau}$ is also obtained.

\begin{prop}\label{prop:invariant_measure_approximation[weak4AC-25]}
   Let Assumptions \ref{assump:A(linear_operator)[weak4AC-25]}-\ref{assump:X_0(Initial Value)[weak4AC-25]} and Assumptions \ref{assump:F_N,tau[weak4AC-25]}, \ref{assume:contractive_or_non-degeneracy[weak4AC-25]}
   hold.
 For $\tau \in (0, \tfrac{1}{2 \widetilde{c}_0})$ and the covariance operator $Q$ being invertible, the full-discretization scheme \eqref{eq:full_discretization[weak4AC-25]} is geometric ergodic, possessing a unique invariant measure $\mu^{N, \tau}$.
   Moreover, 
   there exists constant $C(X_0, Q,q, \varphi, \kappa, \iota) > 0$
   such that for $N \in \N $,
   for any $\varphi \in \mathcal{C}_b^2(H)$, it holds
   \begin{equation}
       \left| 
       \int_{H^N} 
          \varphi 
        \,\mathrm{d} \mu^{N, \tau}
       -
       \int_{H} 
          \varphi 
        \,\mathrm{d} \mu 
        \right|
       \leq
       C(X_0, Q, q, \varphi, \kappa, \iota)
       \left( 
          \tau^{\kappa} 
          + 
          \lambda_N^{-\iota} 
       \right),
   \end{equation}
with $\kappa \in (0, (\gamma \wedge \theta \wedge 1) )$, $\iota \in (0, (\gamma \wedge \rho \wedge 1) )$,
where 
$\gamma$ comes from Assumption \ref{assump:W(noise)[weak4AC-25]} and $\theta, \rho$ are method parameters coming from Assumption
\ref{assump:F_N,tau[weak4AC-25]}.
\end{prop}
\begin{proof}
    According to Corollary \ref{cor:convergence_full-invariant[weak4AC-25]}, it suffices to prove the geometric ergodicity of the full discretization schemes \eqref{eq:full_discretization[weak4AC-25]}. Indeed, using a similar strategy as in Lemma \ref{lem:X^N,tau(L2_bound)[weak4AC-25]}, one derives 
    \begin{equation}
        \begin{aligned}
            \big\| X^{N,\tau}_{t_{m+1}} \big\|^2
            & \leq
            \big\| X^{N,\tau}_{t_m} 
               +
               \tau 
               P_N
               F_{\tau,N}\big( X^{N,\tau}_{t_m} \big)
               +
               P_N \Delta W_m
             \big\|^2  
        \\
            & =
              \big\| 
              X^{N,\tau}_{t_m} 
              \big\|^2 
               +
              \tau^2 
              \big\|
              P_N F_{\tau,N}
              \big( X^{N,\tau}_{t_m} \big)
              \big\|^2
               +
              \| P_N \Delta W_m \|^2   
               +
              2 \tau 
              \langle
              X^{N,\tau}_{t_m}, 
              P_N F_{\tau,N}
              \big( X^{N,\tau}_{t_m} \big)
              \rangle
        \\
        &\qquad
              +
              2 
              \langle
              X^{N,\tau}_{t_m} , P_N\Delta W_m
              \rangle
               +
               2 \tau
               \langle
               P_N
               F_{\tau,N}
               \big( X^{N,\tau}_{t_m} \big),
               P_N \Delta W_m
               \rangle
 \\
     & \leq  (1 -2 \widetilde{c}_0 \tau)
              \big\| 
              X^{N,\tau}_{t_m} 
              \big\|^2 
               +
              \| P_N \Delta W_m \|^2   
              +
              2 
              \langle
              X^{N,\tau}_{t_m} , P_N\Delta W_m
              \rangle
               +
               2 \tau
               \langle
               P_N
               F_{\tau,N}
               \big( X^{N,\tau}_{t_m} \big),
               P_N \Delta W_m
               \rangle.
        \end{aligned}
    \end{equation}
In view of Theorem \ref{thm:X^N,tau(bound)[weak4AC-25]}, recalling the fact that $\sup_{m \in \N_0}\E[  \| X^{N,\tau}_{t_m} \|_V^p ] < \infty, p \geq 1$,
and taking the conditional expectation on both sides of the above inequality, one gets the following Lyapunov structure
\begin{align}
        \E
        \left[  
        \big\| X^{N,\tau}_{t_{m+1}} \big\|^2 
        \Big| \mathcal{F}_{t_m} 
        \right]
     & \leq  (1 -2 \widetilde{c}_0 \tau)
              \big\| 
              X^{N,\tau}_{t_m} 
              \big\|^2 
               +
               \lambda_N^{ \max( 1-\gamma ,0 ) } \|A^{ \frac{\gamma-1}{2} } Q^\frac12 \|^2_{\mathcal{L}_2} \tau  .
\end{align}
Additionally,
since the noise of \eqref{eq:full_discretization[weak4AC-25]} is additive and non-degenerate,
following arguments used in proofs of
\cite[Corollary 3.2]{liuliu2024unique} or \cite[Theorem 3.4]{liu-shen_2025geometric},
one can prove the geometric ergodicity of the full discretization scheme \eqref{eq:full_discretization[weak4AC-25]}.
The desired result follows immediately.
%
\end{proof}

\section{Numerical experiments}
\label{sec:numerical[weak4AC-25]}

In this section, we conduct numerical experiments to support the theoretical results established previously.
In what follows, we consider the following one-dimensional SPDE model:
\begin{equation}\label{eq:ACE_numeri[weak4AC-25]}
   \left\{\begin{array}{l}
\frac{\partial u}{\partial t} (t,x)
=
\frac{\partial^2 u}{\partial x^2} (t,x) + \sigma u(t,x) - u^3(t,x) + \dot{W}(t,x), 
\quad (t,x) \in (0,1]\times(0,1), 
\\
u(0, x) = \sin (\pi x), \quad x \in(0,1),
\\
u(t, 0) = u(t, 1) = 0, \quad t \in(0,1],
\end{array}\right.
\end{equation}
where $\{W(t)\}_{t \in[0, T]}$ is a cylindrical $Q$-Wiener process. 
%
%
In the space-time white noise case (i.e., $Q=I$), Assumption \ref{assump:W(noise)[weak4AC-25]} holds for $\gamma < 1 / 2$. For the trace-class noise case, by choosing $Q$ such that
$
    Q e_1=1, Q e_i=\tfrac{1}{1 + i \log (i)^2} e_i, \forall\  i \geq 2,
$
Assumption \ref{assump:W(noise)[weak4AC-25]},
as well as the non-degeneracy condition in Assumption \ref{assume:contractive_or_non-degeneracy[weak4AC-25]} hold for $\gamma =1$, which can by easily verified by following arguments in \cite[Example 5.3]{kruse2012optimal}.
Evidently, the model \eqref{eq:ACE_numeri[weak4AC-25]} satisfies Assumption \ref{assump:F(Nonlinearity)[weak4AC-25]} with $c_0 = 0.9, c_3 = 1.5$ and $L_f = \sigma$.
For the contractive case $L_f < \lambda_1$, we set $\sigma=1$, and $\sigma =10$ for the non-contractive case $L_f > \lambda_1$.
%
%
%
%
%
%
%
%
We test the scheme \eqref{eq:full_discretization[weak4AC-25]} with $f_{\tau_N}$ given by \eqref{eq:example_f_tau[weak4AC-25]}, where we take $\theta=1, \rho=1, \alpha = \tfrac14, \beta_1=\beta_2=1$.
Throughout the tests,
the expectation is approximated by averaging over $2000$ samples.

\begin{figure}[htp]   
  \centering     
   \subfloat[Long-time behaviors]  
  {
      \includegraphics[width=0.30\textwidth]{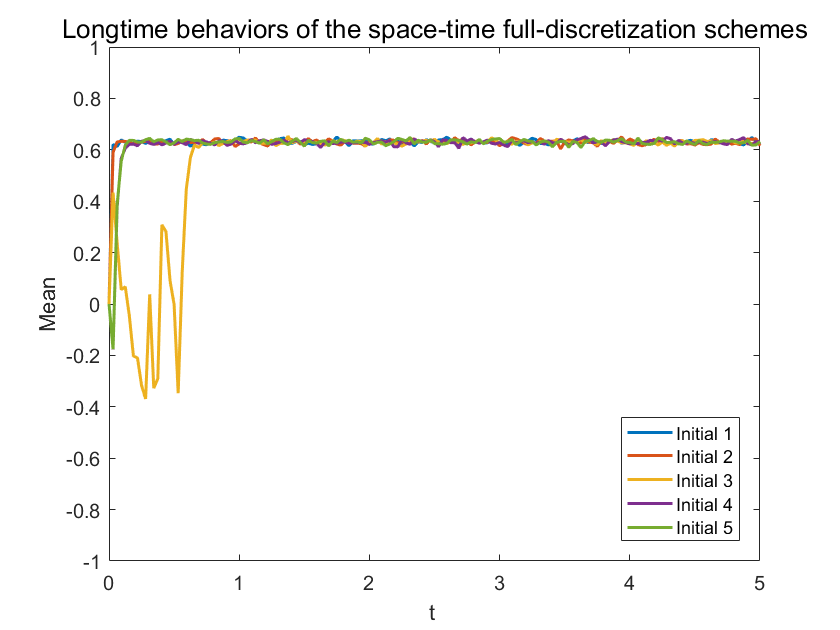}
  }      \hspace{3em}
    \subfloat[Probability densities]  
  {
      \includegraphics[width=0.30\textwidth]{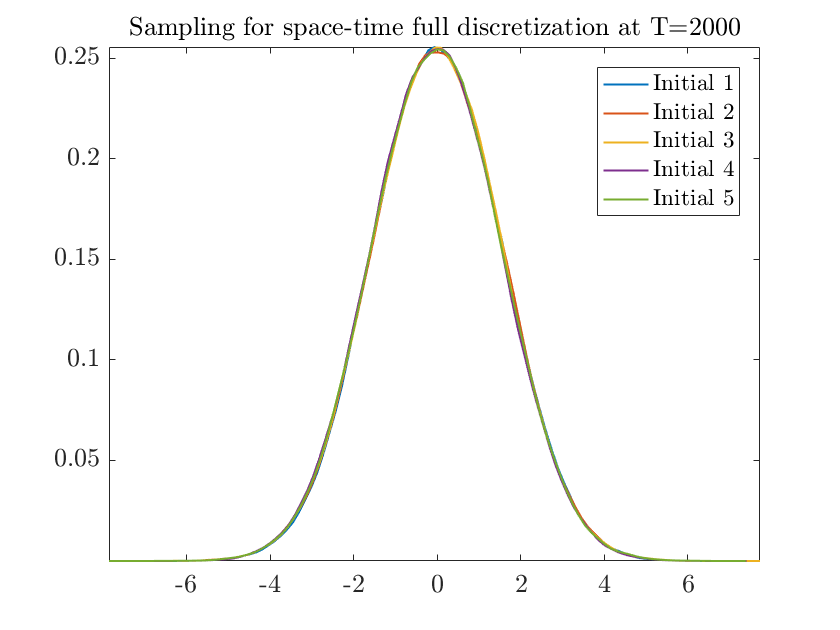}
  }
  \caption{Long-time behaviors and the probability densities for the scheme ($Q=I$)}
  \label{fig:longtime_and_pdf_ex2_I-Wiener[weak4AC-25]} 
\end{figure}

\begin{figure}[htp]   
  \centering     
   \subfloat[Long-time behaviors]  
  {
      \includegraphics[width=0.30\textwidth]{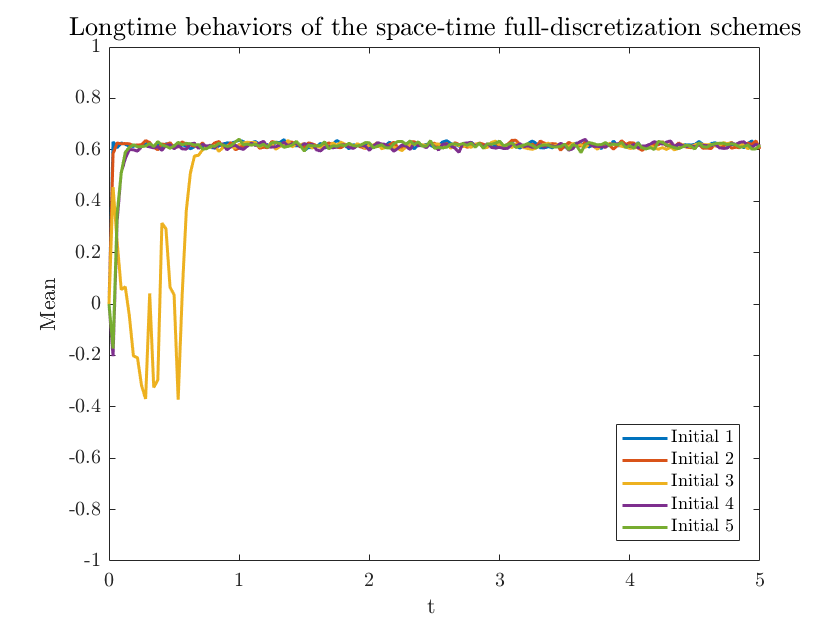}
  }
       \hspace{3em}
    \subfloat[Probability densities]  
  {
      \includegraphics[width=0.30\textwidth]{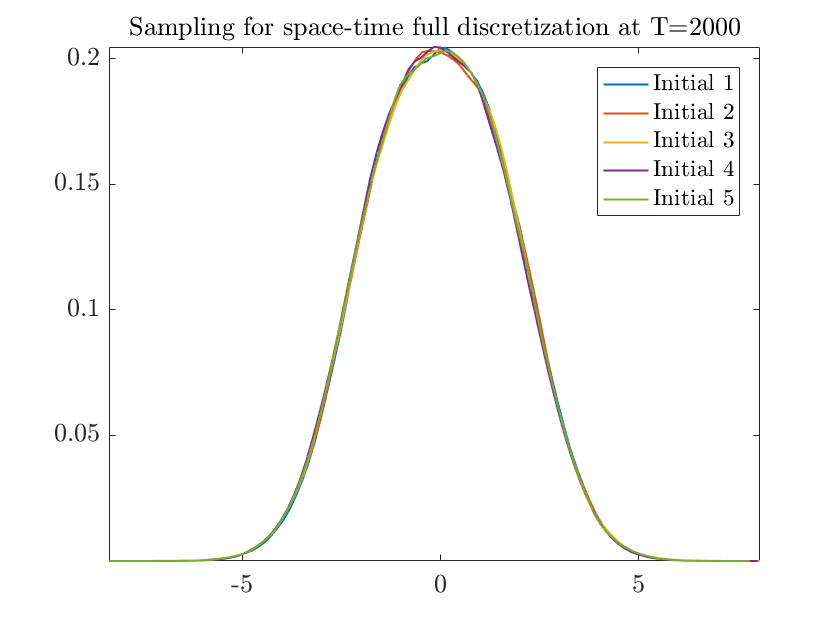}
  }
  \caption{Long-time behaviors and the probability densities for the scheme ($\mathrm{Tr}(Q) < \infty$)}
  \label{fig:longtime_and_pdf_ex2_Q-Wiener[weak4AC-25]} 
\end{figure}

Based on the spatial discretization \eqref{eq:spatial_disc[weak4AC-25]} with $N = 2^{6}$, we first test the long-time behaviors of the proposed scheme.
%
%
%
%
In Figure \ref{fig:longtime_and_pdf_ex2_I-Wiener[weak4AC-25]}
and 
Figure
\ref{fig:longtime_and_pdf_ex2_Q-Wiener[weak4AC-25]},
we show the averages $\E[ \sin ( \| X_m^{N, \tau} \|)]$
started from five different initial values,
where the former is for the contractive SPDE with $\sigma=1$ and driven by cylindrical $I$-Wiener process,
while the latter is for the non-contractive SPDE with $\sigma=10$ and driven by trace-class noise ($\mathrm{Tr}(Q) < \infty$).
As indicated in Figures \ref{fig:longtime_and_pdf_ex2_I-Wiener[weak4AC-25]}, \ref{fig:longtime_and_pdf_ex2_Q-Wiener[weak4AC-25]}, 
the averages $\E[ \sin ( \| X_m^{N, \tau} \|)]$ started
from different initial values converge to the same equilibrium in a short time.
In the same setting,
we also draw samplings for $X_{M}^{N, \tau}$ at $M = T \tau^{-1}$, $T=2000$,
by taking over $5000$ samples,
and depict the associated probability density functions for the first component of $X_{M}^{N, \tau}$ (see right pictures of Figures \ref{fig:longtime_and_pdf_ex2_I-Wiener[weak4AC-25]}, \ref{fig:longtime_and_pdf_ex2_Q-Wiener[weak4AC-25]}).

Next we test the weak convergence rates of the proposed scheme.
%
%
We simulate the weak errors at the endpoint $T=10$.
Particularly, the ``exact" solutions are computed by numerical solutions using sufficiently small stepsizes $N_{\text{exact}} =  2^{12}$ and $ M_{\text{exact}} = 2^{16}$.
%
As shown in Figure \ref{fig:weakerror[weak4AC-25]}, the weak errors of the space-time full-discretizations are depicted on a log-log scale, against $T/M$ with $M = 2^i, i = 9, ...,14$, using three different test functions.
In the left picture of Figure \ref{fig:weakerror[weak4AC-25]}, we show weak errors for the contractive SPDE with $\sigma=1$ and driven by cylindrical $I$-Wiener process. 
The right picture of Figure \ref{fig:weakerror[weak4AC-25]} indicates the weak errors for the non-contractive SPDE with $\sigma=10$ and driven by trace-class noise ($\mathrm{Tr}(Q) < \infty$).
%
%
%
%
%
%
It is shown that the weak convergence rate for the space-time white noise case
is close to $1/2$ in time,
while the weak rate for 
the trace-class noise case is almost $1$ in time.
All the above numerical results confirm the theoretical findings.


\begin{figure}[htp]   
  \centering 
  \subfloat[$Q=I$]
  {
  \includegraphics[width=0.30\textwidth]{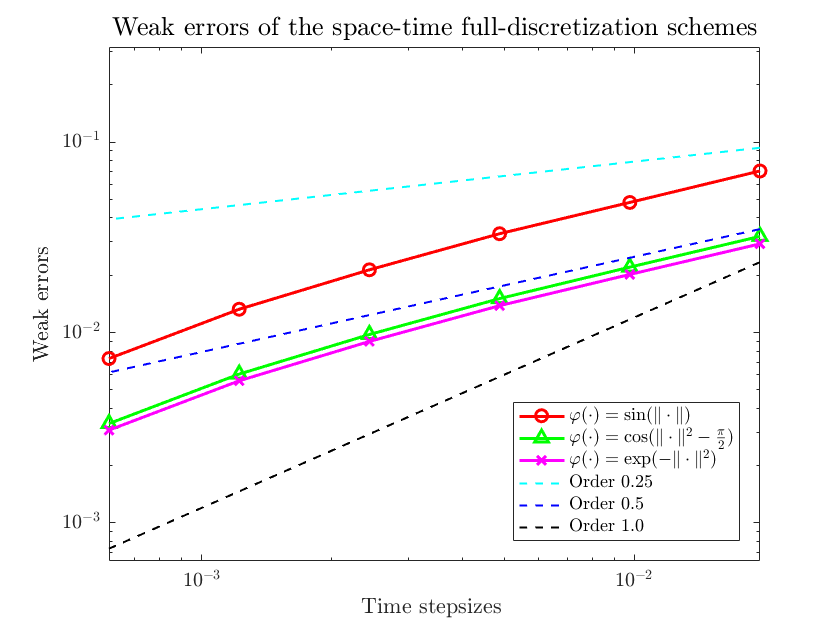}
  }
       \hspace{3em}
  \subfloat[$\mathrm{Tr}(Q) < \infty$]
  {
      \includegraphics[width=0.30\textwidth]{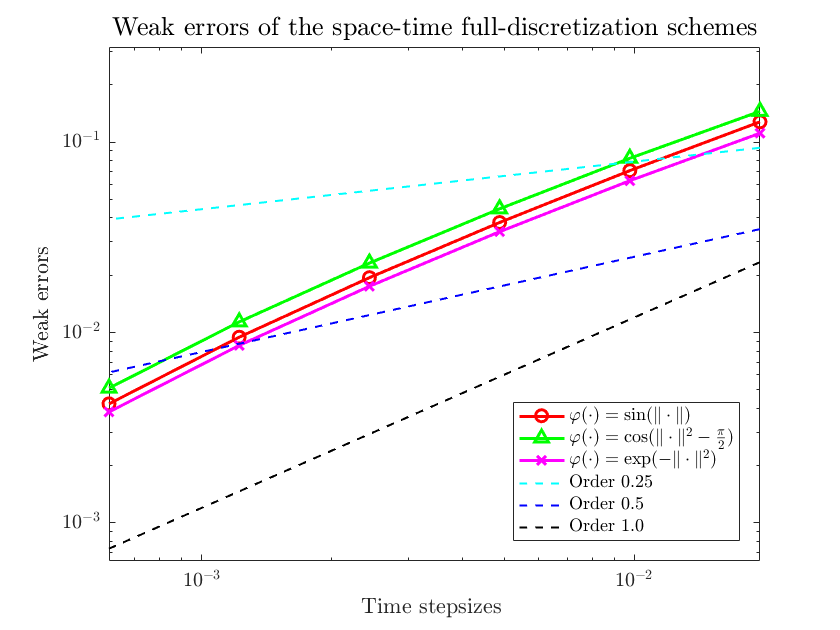}
  }
  \caption{Weak convergence rates
  }
  \label{fig:weakerror[weak4AC-25]} 
\end{figure}

\section{Conclusion and future work}
In this work, we introduce a class of novel full-discretization schemes for long time approximations of parabolic SPDEs. The fully discrete schemes are explicit, easily implementable, and preserve the ergodicity of the original dissipative SPDEs. By fully exploiting a contractive property of the semi-group and the dissipativity of the nonlinearity,
we obtain uniform-in-time moment bounds and uniform-in-time weak convergence rates of the proposed schemes. 
Approximations of the invariant measures are also examined.
We would like to mention that the time-stepping schemes can be also applied to finite element based approximations, whose analysis is, however, more involved and would encounter essential difficulties. This as well as total variation error bounds for the proposed time-stepping schemes are our ongoing projects. Long-time weak approximations of SPDEs with multiplicative noises are also on a list of our future works. 

\vskip6mm
\bibliography{ref}

\end{document}